\documentclass[reqno,12pt]{amsart}
\usepackage[left=1.2in, right = 1.2in, top=1in]{geometry}

\usepackage{amsmath, amssymb, amsthm, mathrsfs, color, textcomp, verbatim}

\usepackage{amstext}
\usepackage{appendix}

\usepackage{cases}
\usepackage{pgfplots}

\usepackage{graphicx}
\usepackage{tikz}
\newtheorem{theorem}{Theorem}[section]
\newtheorem{definition}{Definition}[section]
\newtheorem{remark}{Remark}[section]
\newtheorem{corollary}{Corollary}[section]

\newtheorem{lemma}{Lemma}[section]

\newtheorem{question}{Question}[section]

\usepackage{mathtools}

\usepackage[shortlabels]{enumitem}

\usepackage[square,numbers]{natbib}

\usepackage[colorlinks=true]{hyperref} 

\hypersetup{urlcolor=blue, citecolor=red, linkcolor=blue}

\title{The Critical LYZ equation in K\"ahler Geometry}
\author{Jixiang Fu}
\author{Shing-Tung Yau}
\author{Dekai Zhang}

\address{Shanghai Center for Mathematical Sciences,
	Fudan University,
	Shanghai 200433, China.}
\email{majxfu@fudan.edu.cn}

\address{Yau Mathematical Sciences Center,
Tsinghua University,
Beijing 100084, China.}
\email{styau@tsinghua.edu.cn}
\address{School of Mathematical Sciences, Key Laboratory of Mathematics and Engineering Applications (Ministry of Education), Shanghai Key Laboratory of PMMP, East China Normal University, Shanghai 200241, China. 
}
\email{dkzhang@math.ecnu.edu.cn}

\begin{document}

\begin{abstract}

We establish the existence of smooth solutions for the LYZ equation at the critical phase $\theta =(n-2)\frac{\pi}{2}$, thereby solving the critical case of a problem posed by Collins-Jacob-Yau \cite{cjy2020} and Li \cite{YangLi2022arxiv} concerning the solvability for phase $\theta \le (n-2)\frac{\pi}{2}$.  As applications, we solve the 3-dimensional Hessian equation $\sigma_{2}= 1$ and the 4-dimensional Hessian quotient equation $\sigma_{3}=\sigma_{1}$ under weaker assumptions than previously required \cite{sun2017cpam, szek2018}.


\end{abstract}

\maketitle
\renewcommand{\thefootnote}
\renewcommand{\thefootnote}{\arabic{footnote}} 
\section{Introduction}
In K\"ahler geometry, the Leung-Yau-Zaslow (LYZ) equation, also known as the deformed Hermitian Yang-Mills (dHYM) equation, is a fully nonlinear elliptic partial differential equation arising from two fundamental sources: mirror symmetry \cite{LYZ1999} and supersymmetric gauge theory \cite{MMMS2000}. This equation has become a central object of study at the intersection of geometry and physics with profound connections to stability conditions in derived categories and the Thomas-Yau conjecture \cite{ThomasYau2022CAG}.

Let \((M,\omega)\) be a compact K\"ahler manifold and \(\chi\) a closed real \((1,1)\)-form. Assume that the complex number \(\int_M(\chi + \sqrt{-1}\omega)^n\) is non-zero, and let \(\hat{\theta}\) denote its principal argument. The LYZ equation seeks a smooth function \(u:M\to\mathbb{R}\) such that
\begin{equation}\label{LYZe}
	\begin{aligned}
		\mathrm{Im}\big(e^{-\sqrt{-1}\hat\theta}(\chi_u+\sqrt{-1}\omega)^n\big)=0
	\end{aligned}
\end{equation}
where $\chi_u=\chi+\sqrt{-1}\partial\bar\partial u$. 
Denote by \(\lambda = (\lambda_1,\dots,\lambda_n)\) the eigenvalues of \(\chi_u\) with respect to \(\omega\). The special Lagrangian operator is defined as
\[\theta_\omega(\chi_u) := \sum_{i=1}^n \arctan \lambda_i.\]
Then equation \eqref{LYZe} can be rewritten as
\begin{equation}\label{LYZ2}
	\begin{aligned}
		\theta_{\omega}(\chi_u)=\theta,\quad \theta\equiv n\frac\pi 2-\hat\theta\,(\textup{mod}\,\pi).
	\end{aligned}
\end{equation}
This formulation reveals that the LYZ equation is the complex analogue of the special Lagrangian equation \cite{HarveyLawson1982} for graphs with $\theta_\omega$ playing the role of the Lagrangian angle.
As named by Yuan \cite{yuan2006pams} for the special Lagrangian equation, equation \eqref{LYZ2} is called the \emph{supercritical}  if \(\theta \in ((n-2)\frac{\pi}{2}, n\frac{\pi}{2})\) and the \emph{critical} if \(\theta = (n-2)\frac{\pi}{2}\), i.e., \(\hat{\theta} = \pi\).  

The equation is named after Leung, Yau, and Zaslow \cite{LYZ1999}, who formulated it within the SYZ framework of mirror symmetry, where it serves as the mirror counterpart of the special Lagrangian equation.  Independently, Mari\~ no, Minasian, Moore and Strominger \cite{MMMS2000} derived the same equation from physical considerations in supersymmetric D-brane theory, where the phase parameter $\theta$ encodes the supersymmetry breaking threshold. From this physical perspective, the critical case $\theta =(n-2) \frac{\pi}{2}$ represents the boundary between BPS and non-BPS configurations, thereby providing a natural setting for wall-crossing phenomena.

A key requirement for solving the LYZ equation is the existence of a subsolution introduced by Collins-Jacob-Yau \cite{cjy2020}: a smooth function \(\underline{u}\) satisfying
\begin{align}\label{subsolutiontheta}
	\min\limits_{1\le j\le n}\sum\limits_{\substack{1\le i\le n,\,i\neq j}}\mathrm{arctan}\lambda_i(\chi_{\underline {u}})>\theta-\frac{\pi}{2}.
\end{align}
This notion of subsolution is equivalent to those introduced by Guan \cite{guan2014} and by Sz\'ekelyhidi \cite{szek2018}. Their seminal work plays a crucial role in the existence theory.

The systematic study of the LYZ equation was initiated by Jacob and Yau \cite{jacobyau2017} who solved the two-dimensional case and also obtained partial results in higher dimensions. Collins, Jacob and Yau \cite{cjy2020} then solved the supercritical case under the assumption of a subsolution \(\underline{u}\) together with the extra condition \(\theta_\omega(\chi_{\underline{u}}) > (n-2)\frac{\pi}{2}\). The authors \cite{fyz2024} introduced a geometric flow to solve the equation under the same conditions. The extra condition was subsequently removed by Pingali \cite{Pingali2022APDE} for \(n=3\), by Lin \cite{lin2023adv} for \(n=3,4\), and finally by Lin  \cite{lin2023} in full generality. An alternative proof based on Chen's significant result  \cite{Chen2021Invent} was given by Sun \cite{sun2024}.
For further recent  developments, we refer to
\cite{ChanJacob2023arxiv, ChuCollinsLee2021GT,ChuLee2023Crelle, clt2024jdg, CollinsXieYau2018,CollinsYau2021APDE,HanJin2021CVPDE,HanJin2023Tran,HanJin2024Manu,HuangZhangZhang2022SCM,Jacob2021PAMQ, jacob2022arxiv,JacobSheu2022AsianJM,KhalidDyrefelt2024imrn,Lin2024MRL,Takahashi2021ijm}.

From the perspective of mirror symmetry and stability conditions, the LYZ equation holds a privileged position. As explained in Li's comprehensive survey \cite{YangLi2022arxiv}, the Thomas-Yau conjecture \cite{ThomasYau2022CAG} proposes that the existence of special Lagrangian representatives in a given Hamiltonian isotopy class should be governed by stability conditions in the derived Fukaya category. The LYZ equation emerges as the mirror counterpart of the special Lagrangian equation on the A-side, with the phase parameter $\theta$ playing the role of the stability angle.

Collins, Jacob and Yau \cite{cjy2020} suggested studying the LYZ equation when \(\theta \leq (n-2)\frac{\pi}{2}\). This was later explicitly posed as a question by Li \cite{YangLi2022arxiv}, as resolving it is essential for the LYZ equation to induce a Bridgeland stability condition on \(D^b\mathrm{Coh}(M)\), the bounded derived category of coherent sheaves on \(M\).

In this paper we solve the critical case; that is, we establish the existence of smooth solutions for the critical LYZ equation
\begin{align}\label{LYZ1}
	\theta_{\omega}(\chi_u)=(n-2)\frac{\pi}{2}.
\end{align}
In this case \(\hat\theta=\pi\), and the subsolution condition becomes
\begin{align}\label{subsolution}
	\min\limits_{1\le j\le n}\sum\limits_{\substack{1\le i\le n,\,i\neq j}}^n\mathrm{arctan}\lambda_i(\chi_{\underline {u}})>{(n-3)\frac \pi 2}
\end{align}
which is equivalent to $\theta_{\omega}(\chi_{\underline{u}})>\frac{n}{n-1}(n-3)\frac{\pi}{2}$ and $\mathrm{Im}(\chi_{\underline{u}}+\sqrt{-1}\omega)^{n-1}>0$ (see Lemma \ref{subequi}).
\begin{theorem}\label{FYZ} Let \((M,\omega)\) be a compact K\"ahler manifold and \(\chi\) a closed real \((1,1)\)-form such that the principal argument of \(\int_M(\chi + \sqrt{-1}\omega)^n\) is \(\pi\). If there exists a smooth function \(\underline{u}\) satisfying \eqref{subsolution}, then the critical LYZ equation \eqref{LYZ1} admits a unique smooth solution \(u\) with \(\sup_M u = 0\).
\end{theorem}

The critical and supercritical cases are already known; thus the remaining question is whether singular solutions exist for $\theta<(n-2)\frac\pi 2$. For comparison, in the special Lagrangian equation in \(\mathbb{R}^n\), interior Hessian estimates for critical and supercritical cases were proved by Warren-Yuan \cite{WarrenYuan2010} and Wang-Yuan \cite{WangYuan2014}. For \(\theta < (n-2)\frac{\pi}{2}\), \(C^{1,\alpha}\) solutions were constructed by Nadirashvili-Vladut \cite{NV2010IHP} and Wang-Yuan \cite{wangyuan2013ajm}. Very recently, Mooney-Savin \cite{mooneysavin2023duke} obtained a Lipschitz solution that is not \(C^1\). 

To solve the critical LYZ equation \eqref{LYZ1}, we  consider a family of  equations
\begin{align}\label{LYZt}
	\theta_{\omega}(\chi_{u^{t}}+t\omega)=\theta(t).
\end{align}
We verify that there exists a uniform constant \(C > 0\) such that for small \(t>0\),
\[\theta(t) \in \Bigl((n-2)\frac{\pi}{2} + Ct, (n-2)\frac{\pi}{2} + 2Ct\Bigr)\]
(see Lemma \ref{thetatrange}). Hence each equation in the family \eqref{LYZt} is supercritical. We then prove that the subsolution \(\underline{u}\) of the critical LYZ equation  \eqref{LYZ1}  is also a subsolution of each equation in  \eqref{LYZt}  for small \(t > 0\). Thus we can apply Lin's result \cite{lin2023} (see also Sun \cite{sun2024}) to obtain smooth solutions \(u^t\) to equations \eqref{LYZt} with \(\sup_M u^t = 0\).
To prove Theorem \ref{FYZ}, we establish uniform \(C^{2,\alpha}\) estimates for \(u^t\) that are independent of \(t\), and in particular,  of \((\theta(t) - (n-2)\frac{\pi}{2})^{-1}\).

\begin{theorem}\label{FYZfamily} Assume the same conditions as in Theorem 1.1. Let \(u^t\) be the smooth solution of \eqref{LYZt} with \(\sup_M u^t = 0\). Then there exist uniform constants \(C\) and \(t_0\), depending only on \((M,\omega,\chi)\) and independent of \(t\), such that for any \(t \in (0,t_0)\),
	\begin{equation}
		|u^t|_{C^{2,\alpha}(M)} \le C.
	\end{equation}
\end{theorem}

Before sketching the proof of Theorem \ref{FYZfamily}, it is helpful to recall some important elliptic partial  differential equations related to the LYZ equation.
When $n=2$, the supercritical LYZ equation is equivalent to the complex Monge-Amp\`ere equation 
\begin{align*}
	(\chi_u-\cot\theta\omega)^2=\sin^{-2}\theta\omega^2.
\end{align*}
The complex Monge-Amp\`ere on a compact K\"{a}hler manifold was solved by the second named  author \cite{Yau1976}; the general Hermitian case was  solved by Cherrier \cite{Cherrier1987} and  Guan-Li \cite{GuanLi2010}  partially, and finally by Tosatti-Weinkove \cite{TW2010}. 

When $n=3$, the critical LYZ equation becomes the complex 2-Hessian equation
\begin{align*}
	\sigma_2(\chi_u)=1.
\end{align*}
For the complex $k$-Hessian equation $\sigma_k(\omega_u) = f$ on a compact K\"ahler manifold,
Hou-Ma-Wu \cite{hmw2010} proved the complex Hessian estimate
\begin{align}\label{HMW}
	\sup\limits_{M} |\sqrt{-1}\partial\bar{\partial} u|_{\omega} \le C\bigl(1 + \sup\limits_{M} |\partial u|_{\omega}^2 \bigr)
\end{align}
now known as the Hou-Ma-Wu estimate. They also pointed out that this estimate can be used to obtain the gradient estimate. Later, based on the Hou-Ma-Wu  estimate, 
 Dinew and Ko{\l}odziej, in their breakthrough paper \cite{dk2017ajm}, derived the gradient estimate  by proving a  Liouville-type theorem for the complex Hessian equation \(\sigma_k(\sqrt{-1}\partial\bar\partial u) = 0\) in \(\mathbb{C}^n\), thereby solving the equation $\sigma_k(\omega_u)=f$ on a compact K\"ahler manifold. Sun \cite{sun2017cpam} then solved the equation $\sigma_k(\chi_u)=f$. The Hermitian case was solved independently by Sz\'ekelyhidi \cite{szek2018} and Zhang \cite{zhang2017pjm}.  Sz\'ekelyhidi \cite{szek2018} also derived uniform $C^{2,\alpha}$ estimates for concave  fully nonlinear elliptic partial differential equations. The Hou-Ma-Wu  estimate and the Dinew-Ko{\l}odziej Liouville-type argument have been applied to several important geometric partial differential equations, such as the Gauduchon conjecture which was solved  by Tosatti-Weinkove and Sz\'{e}kelyhidi-Tosatti-Weinkove in a series of significant papers \cite{TW2017, TW2019, STW2018}.
 
 When \(n=4\), the critical LYZ equation is the Hessian quotient equation
 \[\sigma_3(\chi_u) = \sigma_1(\chi_u).\]
 More generally, the Hessian quotient equation is of the form \(\sigma_k(\chi_u) = c\sigma_l(\chi_u)\) with \(0 \leq l < k \leq n\). It reduces to the \(J\)-equation when \(k = n\) and \(l = n-1\), which was solved by Song-Weinkove \cite{songweinkove2008} using the flow method; the case \(k = n, 1 \leq l \leq n-1\) was solved by Fang-Lai-Ma  \cite{flm2011}. The general case was subsequently solved by Sun \cite{sun2017cpam} and Sz\'{e}kelyhidi \cite{szek2018}  via different continuity methods.
 
Now we  briefly outline the proof of Theorem \ref{FYZfamily}. 
The critical LYZ equation has the form 
\[
\sigma_{n-1}(\chi_u) - \sigma_{n-3}(\chi_u) + \sigma_{n-5}(\chi_u) - \cdots = 0.
\]
Compared to the supercritical LYZ equation and the complex Hessian equation, the critical LYZ equation lacks concavity.  To overcome this, we 
introduce  the family of equations (\ref{LYZt}) which  can be written as
\[
\begin{aligned}
&\sigma_{n-1}(\chi_{u^t}+t\omega) - \sigma_{n-3}(\chi_{u^t}+t\omega) + \cdots\\
=&\tan\Bigl(\theta(t)-(n-2)\frac\pi 2\Bigr)(\sigma_{n}(\chi_{u^t}+t\omega) - \sigma_{n-2}(\chi_{u^t}+t\omega) +  \cdots).
\end{aligned}
\]
According to Collins-Jacob-Yau \cite{cjy2020},  we have a uniform $C^0$ estimate independent of $t$; moreover, uniform $C^{2,\alpha}$ estimates follow if one has uniform complex Hessian estimates. However, in \cite{cjy2020} the gradient estimate and the Hou-Ma-Wu complex Hessian estimate depend on $|\cot(\theta(t)-(n-2)\frac\pi 2)|\sim \frac{1}{t}$. Our main contribution is to prove the uniform complex Hessian estimate \eqref{HMW} and the uniform gradient estimate via the Dinew-Ko{\l}odziej Liouville-type argument.

To derive the estimate \eqref{HMW}, we adapt the auxiliary function of Hou-Ma-Wu \cite{hmw2010}. 
Using the subsolution condition and the convexity of the level sets of the Lagrangian phase operator 
(see Section 2 for details), we derive a key inequality independent of \(t\) (see Lemma \ref{keylemma}), 
which then yields the uniform complex Hessian estimate \eqref{HMW}.

For the gradient estimate, the situation is more subtle than the previous works. If the gradient estimate fails, we obtain a  non-constant $C^{0,1}$  function $v$ on \(\mathbb{C}^n\) that, in the sense of currents, is a weak solution to one of the following equations:
\[\sigma_n (\sqrt{-1}\partial\bar \partial v)= 0,\quad \sigma_{n-1}( \sqrt{-1}\partial\bar \partial v)= 0,\quad \text{or}\quad \sigma_{n-1}(\sqrt{-1}\partial\bar \partial v) + \sigma_n (\sqrt{-1}\partial\bar \partial v)= 0.\]
The last case is new and requires a new Liouville-type theorem.
\begin{theorem}
	Let $v: \mathbb{C}^n \rightarrow \mathbb{R}$ be a $C^{0,1}$ function that is $(n-1)$-subharmonic and such that $v + |z_{n+1}|^2$ is $n$-subharmonic in $\mathbb{C}^{n+1}$. Assume there exists a positive constant $C$ with $\|v\|_{C^{0,1}(\mathbb{C}^n)} \le C$ and $\Delta v \le C$ in the weak sense (see Remark~\ref{vlaplacebound} for details). If $v$ is a weak solution in the sense of currents and a viscosity supersolution to the equation
	\begin{align}\label{0nn-1equation}
		\sigma_{n-1}(\sqrt{-1}\partial\bar{\partial} v) + \sigma_{n}(\sqrt{-1}\partial\bar{\partial} v) = 0,
	\end{align}
	then $v$ must be a constant.
\end{theorem}
In the proof of this theorem, we follow the Dinew-Ko{\l}odziej Liouville-type argument as in \cite{dk2017ajm,szek2018} and divide the proof into two cases. The first case can be handled similarly. For the second case, since the equation is not homogeneous, we need to show that the standard mollification \([v]_r\) of \(v\) and \(\frac{1}{3}[(v^\epsilon)^2]_r - \tau_0|z|^2\) are still subsolutions, where \(v^\epsilon = [v]_\epsilon\) and \(\tau_0 > 0\) is sufficiently small. The key observation is that a smooth function \(v\) satisfies \(\sigma_{n-1} + \sigma_n = 0\) (resp. \(\geq 0\)) on \(\mathbb{C}^n\) if and only if the function \(v + |z_{n+1}|^2\) satisfies \(\sigma_n = 0\) (resp. \(\geq 0\)) on \(\mathbb{C}^{n+1}\), where \(z_{n+1}\) is a new complex variable. Based on this observation, we employ both potential theory and viscosity solution methods to obtain the required Liouville-type theorem. As a byproduct, the similar Liouville-type theorem also holds for the equation $\sigma_{k}+\sigma_{k-1}=0$ (see the last section for related problems). 

Finally, we find some applications of the main result.  The first consequence is to solve the 2-Hessian equation in dimension 3 under weaker conditions than those required by Sun \cite{sun2017cpam}.

\begin{corollary} Let \((M,\omega)\) be a compact K\"ahler manifold with dimension \(n=3\) and \(\chi\) a closed real \((1,1)\)-form satisfying \(\chi \wedge \omega > 0\) as a positive \((2,2)\)-form and
	\[3\int_M \chi^2 \wedge \omega = \int_M \omega^3,\quad \int_M \chi^3 < 3\int_M \chi \wedge \omega^2.\]
	Then there exists a unique smooth solution \(u\)  with $\sup_M u = 0$ solving
	\[\sigma_2(\chi_u) = 1,\quad \chi_u \in \Gamma_2(M).\]
\end{corollary}

Another consequence is to solve the Hessian quotient equation in dimension 4 under weaker conditions than those by Sun \cite{sun2017cpam} and Sz\'{e}kelyhidi \cite{szek2018}.

\begin{corollary}\label{4dHQ} Let \((M,\omega)\) be a compact K\"ahler manifold with dimension \(n=4\). Let \(\chi\) be a closed real \((1,1)\)-form satisfying \(3\chi^2 \wedge \omega - \omega^3 > 0\) as a positive \((3,3)\)-form and
	\[\int_M \chi^3 \wedge \omega = \int_M \chi \wedge \omega^3,\quad \int_M \chi \wedge \omega^3 > 0,\quad \operatorname{Re}\int_M (\chi + \sqrt{-1}\omega)^4 < 0.\]
	Then there exists a unique smooth function \(u\) with $\sup_{M} u=0$ solving
	\begin{align*}
		\sigma_3(\chi_u)=\sigma_1(\chi_u),\quad \chi_u\in \Gamma_3(M).
	\end{align*}
\end{corollary}

The organization of this paper is as follows. In Section 2, we provide preliminaries on \(k\)-Hessian operators, the special Lagrangian operator \(\theta_\omega\), and the approximating family \(\theta(t)\). We also give the \(C^0\)-estimate for \(u^t\) using results from \cite{cjy2020}. In Section 3, we first prove a key inequality implied by the subsolution condition, then establish the uniform complex Hessian estimate of Hou-Ma-Wu type which is independent of \((\theta(t) - (n-2)\frac{\pi}{2})^{-1}\). In Section 4, based on the complex Hessian estimate, we prove the Liouville-type theorem, obtain the gradient estimate by the Dinew-Ko{\l}odziej Liouville-type argument, and then prove the main theorems. In Section 5, we apply the main result to the 3-dimensional 2-Hessian equation and the 4-dimensional Hessian quotient equation. In Section 6, we generalize the Liouville-type theorem and raise two questions.

\section{preliminaries }
We first give the definition of the $k$-Hessian operator.
For an $n$-tuple $\lambda = (\lambda_1, \dots, \lambda_n) \in \mathbb{R}^n$ and $1 \le k \le n$, we define
\begin{align}
    \sigma_{k}(\lambda) = \sum_{1 \le i_1 < \cdots < i_{k} \le n} \lambda_{i_1} \cdots \lambda_{i_k}.
\end{align}
For a real $(1,1)$-form $\alpha$, the $k$-Hessian operator of $\alpha$ with respect to the K\"ahler form $\omega$ is defined by
\begin{align}\label{kHessian}
    \sigma_{k}(\alpha) := C_{n}^{k} \frac{\alpha^k \wedge \omega^{n-k}}{\omega^n}.
\end{align}
In local coordinates, writing $\alpha = \sqrt{-1} \alpha_{i\bar j} \, dz_i \wedge d\bar z_j$ and $\omega = \sqrt{-1} g_{i\bar j} \, dz_i \wedge d\bar z_j$, we have
\[
\sigma_k(\alpha) = \sigma_k\!\bigl(\lambda(g^{j\bar l} \alpha_{i\bar l})\bigr).
\]
For simplicity, we also denote $\sigma_k\!\bigl(\lambda(g^{j\bar l} \alpha_{i\bar l})\bigr)$ by $\sigma_k(g^{j\bar l} \alpha_{i\bar l})$.

We need the following convex cones:
\begin{align*}
    \Gamma_{k} &:= \{\lambda \in \mathbb{R}^n \mid \sigma_{m}(\lambda) > 0,\ 1 \le m \le k\}, \\
    \Gamma_{k}(M) &:= \{\alpha \mid \alpha\ \text{is a smooth real $(1,1)$-form on $M$ with } \sigma_{m}(\alpha) > 0,\ 1 \le m \le k \}, \\
    \mathcal{M}(\Gamma_{k}) &:= \{A \mid \text{$A$ is an $n \times n$ Hermitian matrix and } \lambda(A) \in \Gamma_{k} \}.
\end{align*}
The Hessian operators possess the following concavity property, which plays an important role in the proofs of the Liouville-type theorem and the gradient estimate in Section 4.
\begin{lemma}\label{concave}{(G\r{a}rding \cite{Garding1959}, Huisken-Sinestrari \cite{HuiskenSinestrari1999acta})}
    $\frac{\sigma_{k}(\lambda)}{\sigma_{k-1}(\lambda)}$ and $\sigma_{k-1}^{\frac{1}{k-1}}(\lambda)$ are concave in  $\Gamma_{k-1}$ with $2\le k\le n$. Moreover, $\frac{\sigma_{k}(A)}{\sigma_{k-1}(A)}$ and  $\sigma_{k-1}^{\frac{1}{k-1}}(A)$ are concave in $\mathcal{M}(\Gamma_{k-1})$.
\end{lemma}
The concavity of $\sigma_{k-1}^{1/(k-1)}$ in $\Gamma_{k-1}$ was proved by G\r{a}rding \cite{Garding1959}; the concavity of $\frac{\sigma_{k}}{\sigma_{k-1}}$ in $\Gamma_{k-1}$ was proved by Huisken--Sinestrari \cite{HuiskenSinestrari1999acta}.

Consider the Lagrangian phase operator
\begin{align}\label{070601}
    \theta(\lambda) := \sum_{i=1}^n \arctan\lambda_i \quad \text{for } \lambda = (\lambda_1, \dots, \lambda_n) \in \mathbb{R}^n,
\end{align}
and its level set
\begin{align*}
    \Gamma^{\tau} := \{\lambda \in \mathbb{R}^n \mid \theta(\lambda) > \tau\} \subset \mathbb{R}^n
    \quad \text{for } \tau \in \bigl[(n-2)\tfrac{\pi}{2}, n\tfrac{\pi}{2}\bigr).
\end{align*}
The following lemma collects some key properties of the Lagrangian phase operator due to Yuan \cite{yuan2006pams} and Wang-Yuan \cite{WangYuan2014}. 
\begin{lemma}[Yuan \cite{yuan2006pams}, Wang-Yuan \cite{WangYuan2014}]\label{yuanlemma}
    Assume that \(\theta(\lambda) \geq \tau\) for some \(\tau \in [(n-2)\frac{\pi}{2}, n\frac{\pi}{2})\), and that the components of \(\lambda = (\lambda_1, \dots, \lambda_n)\) satisfy \(\lambda_1 \ge \lambda_2 \ge \cdots \ge \lambda_n\). Then the following hold:
    \begin{enumerate}[(i)]
        \item \(\lambda_1 \ge \lambda_2 \ge \cdots \ge \lambda_{n-1} > 0\) and \(\lambda_{n-1} \ge |\lambda_n|\);
        \item \(\lambda_1 + (n-1)\lambda_n \ge 0\);
        \item \(\lambda \in \Gamma_{n-1}\);
        \item \(\Gamma^{\tau}\) is convex.
    \end{enumerate}
\end{lemma}

For the supercritical LYZ equation, Collins, Jacob and Yau \cite{cjy2020} proved that if $\theta(\chi_{\underline u})>(n-2)\frac\pi 2$, the subsolution \eqref{subsolutiontheta} is equivalent to $\mathrm{Re}(\chi_{\underline u}+\sqrt{-1}\omega)^{n-1}-\cot\hat\theta\mathrm{Im}(\chi_{\underline u}+\sqrt{-1}\omega)^{n-1}>0$.
For the critical LYZ equation, we obtain the following equivalent formulation.
\begin{lemma}\label{subequi}
	The subsolution \eqref{subsolution} is equivalent to the following 
	\begin{align}
		&\theta(\chi_{\underline{u}})>\frac{n}{n-1}
	\frac{n-3}{2}\pi,\label{subsum}\\
	&\mathrm{Im}(\chi_{\underline u}+\sqrt{-1}\omega)^{n-1}>0.\label{subn-1}
	\end{align}
\end{lemma}

	\begin{proof}
		Let $\mu=(\mu_1,\cdots,\mu_n)$ be the eigenvalues of $\chi_{\underline{u}}$ with respect to $\omega$, ordered so that $\mu_1\ge \cdots\ge \mu_n$.
		Define
		\[
		\theta_{i}=\mathrm{arccot}\mu_{i}\in (0, \pi),
		\]
		which is the principal argument of $\mu_i+\sqrt{-1}$.
		
 We first prove that
	the subsolution condition implies \eqref{subsum} and \eqref{subn-1}.
		By the subsolution condition, for every $j$,
		\[
		\sum_{\substack{1\le i\le n,\,i\neq j}}\arctan\mu_i>(n-3)\frac{\pi}{2}.
		\]
		Summing these inequalities over $j=1,\dots,n$ yields \eqref{subsum}.
		
		Since $\theta_i=\frac{\pi}{2}-\arctan\mu_i$, the subsolution is equivalent to
		\[
		\sum_{\substack{1\le i\le n,\,i\neq j}}\theta_i<\pi,\quad \forall\ 1\le j\le n.
		\]
	Hence the principal argument of  $\prod_{\substack{1\le i\le n\\ i\neq j}}(\mu_i+\sqrt{-1})$ is $\sum_{\substack{1\le i\le n\\ i\neq j}}\theta_i\in(0,\pi)$, so its imaginary part is positive. This proves \eqref{subn-1}.
		
		Conversely, we show that  conditions \eqref{subsum} and \eqref{subn-1} imply the subsolution \eqref{subsolution}.		
		From  $\theta_1\le\cdots\le\theta_n$ we obtain 
		\[
		\theta_n\ge \frac{1}{n-1}\sum_{i=2}^{n-1}\theta_i.
		\]
		Consequently,
		\[
		\sum_{i=2}^{n-1}\theta_i
		\le \frac{n-1}{n}\sum_{i=2}^{n}\theta_i
		<\frac{n-1}{n}\sum_{i=1}^{n}\theta_i<\pi,
		\]
		where the last inequality follows from \eqref{subsum}.
		Thus
		\[
		\mathrm{Im}\prod_{i=2}^{n-1}(\mu_i+\sqrt{-1})>0.
		\]
		Using \eqref{subn-1} we have
		\[
		0<\mathrm{Im}\prod\limits_{i=2}^{n}(\mu_i+\sqrt{-1})
		=\mu_n \mathrm{Im}\prod_{i=2}^{n-1}(\mu_i+\sqrt{-1})+ \mathrm{Re}\prod_{i=2}^{n-1}(\mu_i+\sqrt{-1}).
		\]
		This implies
		\[
		\cot\Bigl(\sum_{i=2}^{n-1}\theta_i\Bigr)>\cot(\pi-\theta_n).
		\]
		Since $\cot$ is strictly decreasing on $(0,\pi)$ and $\sum\limits_{i=2}^{n-1}\theta_i<\pi$, we deduce
		\[
		\sum_{i=2}^{n-1}\theta_i<\pi-\theta_n
		\quad \text{i.e.} \quad
		\sum_{i=2}^{n}\theta_i<\pi,
		\]
	which proves the subsolution condition for $j=1$, the ordering then gives it for all $j$. 
	\end{proof}
	
Let $\hat \theta(t)$ denote the principal argument of the complex number $\int_M (\chi + t\omega + \sqrt{-1}\omega)^n$, and set $\theta(t) =\frac{n\pi}{2} - \hat \theta(t)$. We will show that $\hat \theta(t)$ is close to $\pi$ when $t>0$ is sufficiently small.
\begin{lemma}\label{thetatrange}
    Assume that $\underline u$ is a subsolution of the critical LYZ equation \eqref{LYZ1}. Then there exist uniform positive constants $t_0$ and $C$, depending only on $\chi$ and $\omega$, such that for all $t \in (0, t_0)$,
    \[
    \hat \theta(t) \in \bigl( \pi - 2Ct,\; \pi - Ct \bigr).
    \]
\end{lemma}

\begin{proof}
    By a direct computation, we have
    \[
    (\chi + t\omega + \sqrt{-1}\omega)^n = (\chi + \sqrt{-1}\omega)^n 
    + t\, n (\chi + \sqrt{-1}\omega)^{n-1} \wedge \omega + O(t^2).
    \]
    Since \( \mathrm{Re}\int_M (\chi + \sqrt{-1}\omega)^n < 0\), it follows that, for sufficiently small \(t\), there exists a constant $C_0>0$ such that 
    \[
    \mathrm{Re}\int_M (\chi + t\omega + \sqrt{-1}\omega)^n 
    < \mathrm{Re}\int_M (\chi + \sqrt{-1}\omega)^n + C_0t < 0.
    \]

    We now show that the imaginary part is positive. The subsolution condition implies the positivity
    \[
    \mathrm{Im}\, (\chi_{\underline{u}} + \sqrt{-1}\omega)^{n-1} > 0.
    \]
    Hence,
    \begin{align*}
        \mathrm{Im}\int_M (\chi + \sqrt{-1}\omega)^{n-1} \wedge \omega 
        = \mathrm{Im}\int_M (\chi_{\underline{u}} + \sqrt{-1}\omega)^{n-1} \wedge \omega > 0.
    \end{align*}
    Consequently, for $t$ small enough
    \begin{align*}
        \mathrm{Im}\int_M (\chi + t\omega + \sqrt{-1}\omega)^n 
        &= t\, \mathrm{Im}\int_M (\chi + \sqrt{-1}\omega)^{n-1} \wedge \omega + O(t^2) > 0.
    \end{align*}

    Therefore, there exists a uniform constant \(C > 0\), depending only on \(\chi\) and \(\omega\), such that the principal argument \(\hat\theta(t)\) of \(\int_M (\chi + t\omega + \sqrt{-1}\omega)^n\) satisfies
    \[
    \hat\theta(t) = \pi - \arctan\!\Bigl( 
        \frac{t\, \mathrm{Im}\int_M (\chi + \sqrt{-1}\omega)^{n-1} \wedge \omega + O(t^2)}
        {-\mathrm{Re}\int_M (\chi + \sqrt{-1}\omega)^n + O(t)} 
    \Bigr) \in (\pi - 2Ct,\; \pi - Ct)
    \]
    for all sufficiently small \(t > 0\).
\end{proof}

Hence, for sufficiently small $t>0$, Lemma \ref{thetatrange} together with the definition 
$\theta(t)=\frac{n\pi}2-\hat\theta(t)$ gives
\begin{align}\label{thetat}
    \theta(t) \in \bigl((n-2)\tfrac{\pi}{2} + Ct,\; (n-2)\tfrac{\pi}{2} + 2Ct \bigr),
\end{align}
and consequently, 
\begin{align*}
    \min_{1\le j\le n} \sum_{\substack{1\le i\le n,\,i\neq j}} \arctan\lambda_i(\chi_{\underline{u}} + t\omega)
    > (n-3)\frac{\pi}{2} + 2Ct > \theta(t) - \frac{\pi}{2}.
\end{align*}
This shows that $\underline{u}$ remains a subsolution of equation \eqref{LYZt} for all sufficiently small $t>0$.

By the result of Lin \cite{lin2023} (see also Sun \cite{sun2024}), for each such $t$, there exists a smooth solution $u^{t}$ to the supercritical LYZ equation \eqref{LYZt}. Moreover, Proposition 3.6 of \cite{cjy2020} provides a uniform $C^0$ estimate for these solutions.

\begin{lemma}[Collins-Jacob-Yau \cite{cjy2020}]
    Let $\underline{u}$ be a subsolution of the critical LYZ equation \eqref{LYZ1}, and let $u^t$ be the solution of \eqref{LYZt} with $\sup_M u^t = 0$. Then there exist constants $t_0>0$ and $C>0$, depending only on the background data $(\chi,\omega)$ and $\underline{u}$, such that for all $t \in (0, t_0)$,
    \[
    \| u^{t} \|_{C^0(M)} \le C.
    \]
\end{lemma}

 \section{Complex Hessian estimates}

We first collect some useful formulas for the complex Hessian estimate from \cite{fyz2024}.
\begin{lemma}
    Let $u^{t}$ be a solution of the LYZ equation \eqref{LYZt}. Then:
    \begin{enumerate}
        \item[(i)] The linearized operator $\mathcal{L}$ of \eqref{LYZt} takes the form
        \[
        \mathcal{L}(v) = \sum_{i,j}F^{i\bar{j}} v_{i\bar{j}},
        \]
        where the matrix $(F^{i\bar{j}})$ is the inverse of $w g^{-1} w + g$. Here $g = (g_{i\bar{j}})$ and $w = (w_{i\bar{j}})$ with $w_{i\bar{j}} = \chi_{i\bar{j}} + t g_{i\bar{j}} + u^{t}_{i\bar{j}}$.
        
        \item[(ii)] Differentiating the LYZ equation \eqref{LYZt} (using covariant derivatives) yields
        \begin{align}
            &\sum_{i, j}F^{i\bar{j}} w_{i\bar{j}, p} = 0, \label{differentiate1} \\
            &\sum_{i, j}F^{i\bar{j}} w_{i\bar{j}, p\bar{p}} =\sum_{i, j, k, l} F^{i\bar{l}} F^{k\bar{j}} w_{i\bar{j}, p} \bigl( w_{k\bar{m}, \bar{p}} g^{r\bar{m}} w_{r\bar{l}} + w_{k\bar{m}} g^{r\bar{m}} w_{r\bar{l}, \bar{p}} \bigr). \label{differentiate2}
        \end{align}
    \end{enumerate}
\end{lemma} 	
  	
The subsolution condition plays a crucial role in the $C^2$-estimates for solutions of fully nonlinear elliptic s, as demonstrated in the works of Guan \cite{guan2014} and Sz\'ekelyhidi \cite{szek2018}. Following the approach in these references, we establish key inequalities that will be essential for the complex Hessian estimate. A favorable feature is that the constant $\delta_0$ appearing below can be chosen to be uniform.

\begin{lemma}\label{keylemma}
	Let $\mu \in \mathbb{R}^n$ satisfy
	\[
	A(\mu) := \min_{1 \le j \le n} \sum_{i \neq j} \arctan \mu_i \; > \; A_0 \; > \; (n-3)\frac{\pi}{2}.
	\]
	Then, for every $\lambda \in \mathbb{R}^n$ with $\theta(\lambda) = \theta(t)$, there exist positive constants $t_0$ and $\delta_0$, depending only on $n$, $\mu$, $A_0$, and the constant $C$ from Lemma \ref{thetatrange}, such that for any $t \in (0, t_0)$, one of the following statements holds:
	\begin{enumerate}[(i)]
		\item $\displaystyle\sum_{i=1}^n \frac{\mu_i - \lambda_i}{1+\lambda_i^2} \ge \delta_0 \sum_{i=1}^n \frac{1}{1+\lambda_i^2}$;\quad or
		\item 
		 $\displaystyle\frac{1}{1+\lambda_j^2} \ge \delta_0 \sum_{i=1}^n \frac{1}{1+\lambda_i^2}$ for all $j =1,\dots,n$.
	\end{enumerate}
\end{lemma}  

\begin{proof}
Given that $A(\mu) > A_0 > (n-3)\frac{\pi}{2}$, we can choose a small positive constant $\delta > 0$ such that 
$$A(\mu - 2\delta \mathbf{1}) > \frac{A_0 + (n-3)\frac{\pi}{2}}{2},$$
where $\mathbf{1} = (1, \ldots, 1) \in \mathbb{R}^n$.

For $\lambda \in \mathbb{R}^n$ satisfying $\theta(\lambda) = \theta(t)$, assume without loss of generality that $\lambda_1 \ge \cdots \ge \lambda_n$. Recall that the vector 
$$\mathbf{n} = \Bigl( \frac{1}{1+\lambda_1^2}, \cdots, \frac{1}{1+\lambda_n^2} \Bigr)$$
is an inward normal to the hypersurface $\partial\Gamma^{\theta(t)}$ at $\lambda$. Let $T_{\lambda}$ denote the tangent plane to  $\partial\Gamma^{\theta(t)}$ at $\lambda$.

\textbf{Case 1:} $\bigl((\mu - \delta\mathbf{1}) - \lambda\bigr) \cdot \mathbf{n} \geq 0$.
In this case, we have
\begin{align}\label{case1}
\sum_{i=1}^n \frac{\mu_i - \lambda_i}{1 + \lambda_i^2} \ge \delta \sum_{i=1}^n \frac{1}{1 + \lambda_i^2}.
\end{align}

\textbf{Case 2:} $\bigl((\mu - \delta\mathbf{1}) - \lambda\bigr) \cdot \mathbf{n} < 0$.
We claim that there exists a positive constant $s$ such that $\mu - \delta\mathbf{1} + s e_1 \in \Gamma^{\theta(t)}$.
Indeed, observe that
\begin{align*}
\theta(\mu - \delta\mathbf{1} + s e_1)
&= \arctan (\mu_1 - \delta + s) + \sum_{i=2}^n \arctan (\mu_i - \delta) \\
&> \arctan (\mu_1 - \delta + s) + A(\mu - \delta\mathbf{1}) \\
&\ge \arctan (\mu_1 - \delta + s) + \frac{1}{2}\Bigl(A_0 + (n-3)\frac{\pi}{2}\Bigr).
\end{align*}
To ensure $\theta(\mu - \delta\mathbf{1} + s e_1) > \theta(t)$, it suffices to choose $s$ satisfying
\[
\arctan (\mu_1 - \delta + s) > \theta(t) - \frac{1}{2}\Bigl(A_0 + (n-3)\frac{\pi}{2}\Bigr).
\]
Take
\[
s = |\delta - \mu_1| + \tan\Bigl(\frac{\pi}{2} - \frac{1}{4}\bigl(A_0 - (n-3)\frac{\pi}{2}\bigr)\Bigr),
\]
and assume $0 < t < t_0$ with $t_0 = \frac{1}{8C}\bigl(A_0 - (n-3)\frac{\pi}{2}\bigr)$. Then the desired inequality holds.

The convexity of $\Gamma^{\theta(t)}$ implies that the points $\mu - \delta\mathbf{1}$ and $\mu - \delta\mathbf{1} + s \mathbf{e}_1$ lie on opposite sides of the tangent plane $T_{\lambda}$.
Hence, there exists $a_0 \in (0, 1)$ such that the convex combination
\[
a_0(\mu - \delta\mathbf{1} + s \mathbf{e}_1) + (1-a_0)(\mu - \delta\mathbf{1}) = \mu - \delta\mathbf{1} + a_0 s \mathbf{e}_1
\]
lies on $T_{\lambda}$.

We now show that $\mu - \delta\mathbf{1} + a_0 s \mathbf{e}_1 + s \mathbf{e}_1 - \delta \mathbf{e}_n \in \Gamma^{\theta(t)}$. Indeed,
\begin{align*}
&\theta\bigl(\mu - \delta\mathbf{1} + a_0 s \mathbf{e}_1 + s \mathbf{e}_1 - \delta \mathbf{e}_n\bigr) \\
&= \arctan (\mu_1 - \delta + a_0 s + s) + \sum_{i=2}^{n-1} \arctan (\mu_i - \delta) + \arctan (\mu_n - 2\delta) \\
&\ge \arctan (\mu_1 - \delta + a_0 s + s) + A(\mu - 2\delta\mathbf{1}) \\
&\ge \arctan (\mu_1 - \delta + s) + \frac{1}{2}\Bigl(A_0 + (n-3)\frac{\pi}{2}\Bigr) \\
&> \theta(t).
\end{align*}
Thus, we have
\begin{align*}
0 &\le \bigl((\mu - \delta\mathbf{1} + a_0 s \mathbf{e}_1 + s \mathbf{e}_1 - \delta \mathbf{e}_n) - (\mu - \delta\mathbf{1} + a_0 s \mathbf{e}_1)\bigr) \cdot \mathbf{n} \\
&= (s \mathbf{e}_1 - \delta \mathbf{e}_n) \cdot \mathbf{n} \\
&= \frac{s}{1+\lambda_1^2} - \frac{\delta}{1+\lambda_n^2},
\end{align*}
which is equivalent to
\begin{equation}\label{eq:ineq_lambda1_lambdan}
\frac{1}{1+\lambda_1^2} \ge \frac{\delta}{s(1+\lambda_n^2)}.
\end{equation}

By Lemma \ref{yuanlemma}, we have $\lambda_1 \ge \cdots \ge \lambda_{n-1} \ge |\lambda_n|$. Consequently, for each $j = 1, \dots, n$,
\[
\frac{1}{1+\lambda_n^2} \ge \frac{1}{1+\lambda_j^2}.
\]
 Combining with \eqref{eq:ineq_lambda1_lambdan}, we obtain 
\[
\frac{1}{1+\lambda_1^2} \ge \frac{\delta}{s(1+\lambda_j^2)}, \quad \textup{for all $1\leq j\leq n$.}
\]
Summing over $j = 1, \dots, n$ yields
\[
\frac{n}{1+\lambda_1^2} \ge \frac{\delta}{s} \sum_{j=1}^n \frac{1}{1+\lambda_j^2},
\]
and therefore
\[
\frac{1}{1+\lambda_1^2} \ge \frac{\delta}{n s} \sum_{j=1}^n \frac{1}{1+\lambda_j^2}.
\]
Since $\lambda_1^2 \ge \lambda_j^2$ for all $j$, we have 
\begin{align}\label{case2}
\frac{1}{1+\lambda_j^2} \ge \frac{1}{1+\lambda_1^2} \ge \frac{\delta}{n s} \sum_{i=1}^n \frac{1}{1+\lambda_i^2}.
\end{align}

In view of \eqref{case1} and \eqref{case2}, choosing $\delta_0 = \frac{\delta}{n s}$ completes the proof.
\end{proof}

By the preceding lemma, we obtain the following key result, noting that the estimate does not depend on $(\theta(t)-(n-2)\frac{\pi}{2})^{-1}$.
\begin{lemma}\label{keylemma1}
    Let $\underline{u}$ be a subsolution of the critical LYZ equation \eqref{LYZ1} and $u^t$ the solution of the LYZ equation \eqref{LYZt}. 
    Then there exist uniform positive constants $t_0$ and $\delta_0$ such that for any $t \in (0, t_0)$ and any $z \in M$, one of the following holds at $z$:
    \begin{enumerate}[label=(\roman*)]
        \item $\displaystyle\sum_{i,j} F^{i\bar{j}} (\underline{u}_{i\bar{j}} - u^{t}_{i\bar{j}}) \ge \delta_0 \sum_{i,j} F^{i\bar{j}} g_{i\bar{j}}$, or
        \item $\displaystyle\mathcal{F}_{\min} \ge \delta_0 \sum_{i,j} F^{i\bar{j}} g_{i\bar{j}}$,
    where $\mathcal{F}_{\min}$ denotes the smallest eigenvalue of the matrix $(F^{i\bar{k}} g_{j\bar{k}})_{n \times n}$. 
   \end{enumerate}
\end{lemma}
 
\begin{proof}
Let $\mu = (\mu_1, \dots, \mu_n)$ be the eigenvalues of $\chi_{\underline{u}}$ with respect to $\omega$, arranged so that $\mu_1 \ge \cdots \ge \mu_n$.

Fix $z \in M$. We choose normal coordinates at $z$ such that
\[
g_{i\bar{j}} = \delta_{ij}, \quad 
\chi_{t, u^{t}} = \sum_{i=1}^n \lambda_i \sqrt{-1} \, dz^i \wedge d\bar{z}^i \quad \text{with} \quad \lambda_1 \ge \cdots \ge \lambda_n,
\]
where $\chi_{t, u^{t}} = \chi + t\omega + \sqrt{-1} \partial \bar{\partial} u^t$.
In these coordinates, at $z$ we have
\begin{align}
F^{i\bar{j}} = \frac{\delta_{ij}}{1+\lambda_i^2}.
\end{align}

Let $\delta_0$ and $t_0$ be the uniform constants given by Lemma \ref{keylemma}. We now prove the lemma by distinguishing two cases.

\textbf{Case 1:} $\displaystyle\sum_{i} F^{i\bar{i}}(\mu_i - \lambda_i) \ge \delta_0 \sum_{i} F^{i\bar{i}}$.

Since $F^{1\bar{1}} \le \cdots \le F^{n\bar{n}}$ and $\mu_1 \ge \cdots \ge \mu_n$, the Schur--Horn theorem \cite{Horn1954AJM} implies
\begin{align*}
\sum_{i,j} F^{i\bar{j}} (\underline{u}_{i\bar{j}} - u^t_{i\bar{j}}) 
&= \sum_{i,j} F^{i\bar{j}} \bigl( (\chi_{\underline{u}})_{i\bar{j}} - (\chi_{t, u^t})_{i\bar{j}} \bigr) + t \sum_{i} F^{i\bar{i}} \\
&\ge \sum_{i,j} F^{i\bar{j}} \bigl( (\chi_{\underline{u}})_{i\bar{j}} - (\chi_{t, u^t})_{i\bar{j}} \bigr) \\
&= \sum_{i} F^{i\bar{i}} (\mu_i - \lambda_i) \\
&\ge \delta_0 \sum_{i} F^{i\bar{i}}.
\end{align*}

\textbf{Case 2:} $\displaystyle\sum_{i} F^{i\bar{i}} (\mu_i - \lambda_i) < \delta_0 \sum_{i} F^{i\bar{i}}$.

By Lemma \ref{keylemma1}, we obtain
\[
\mathcal{F}_{\min} = \frac{1}{1+\lambda_1^2} \ge \delta_0 \sum_{i} \frac{1}{1+\lambda_i^2}
= \delta_0 \sum_{i,j} F^{i\bar{j}} g_{i\bar{j}}.
\]

This completes the proof.
\end{proof}
  	
We use the auxiliary function introduced by Hou-Ma-Wu \cite{hmw2010}, who established complex Hessian estimates for the complex $k$-Hessian equation on a compact K\"ahler manifold, to derive the complex Hessian estimates for families of LYZ equations.
Here, instead of using the largest eigenvalue of $(g^{j\bar k} w_{i\bar k})_{n\times n}$ in \cite{cjy2020}, we employ the term $w_{i\bar j} \xi^{i} \bar{\xi}^{j}$, which simplifies the proof considerably. 

\begin{theorem}\label{thm: Hessian estimate}
    Let $\underline{u}$ be a subsolution of the critical LYZ equation \eqref{LYZ1} and let $u^t$ be the solution of the LYZ equation \eqref{LYZt} with $\sup_M u^t = 0$. 
    Then there exist uniform positive constants $t_0$ and $C$, independent of $\bigl(\theta(t)-(n-2)\frac{\pi}{2}\bigr)^{-1}$, such that for any $t \in (0, t_0)$,
    \begin{align}\label{eq: Hessian estimate}
        \sup_M |\sqrt{-1}\partial\bar{\partial} u^t|_{\omega} \le C \bigl(1 + \sup_M |\nabla u^t|_{\omega}^2 \bigr).
    \end{align}
\end{theorem}

 \begin{proof}
Without loss of generality, we assume $\underline{u} = 0$.
Let $w_{i\bar{j}} = \chi_{i\bar{j}} + t g_{i\bar{j}} + u^{t}_{i\bar{j}}$. For simplicity, we write $u$ instead of $u^t$ in the following proof.  
Denote $S(T^{1,0}M) := \bigcup_{z \in M} \{ \xi \in T^{1,0}_z M \mid |\xi|_{\omega} = 1 \}$. We consider the auxiliary function on $S(T^{1,0}M)$:
\[
H(z,\xi) = \log\bigl(w_{k\bar{l}} \xi^k \bar{\xi}^l\bigr) + \varphi\bigl(|\nabla u|_g^2\bigr) + \psi(u),
\]
where $\varphi$ and $\psi$ are defined by
\[
\begin{aligned}
\varphi(s) &= -\frac{1}{2}\log\!\left(1 - \frac{s}{2K}\right), \quad 0 \le s \le K-1,\\[2mm]
\psi(s) &= -A \log\!\left(1 + \frac{s}{2L}\right), \quad -L+1 \le s \le 0,
\end{aligned}
\]
with
\[
K = \sup_M |\nabla u|_g^2 + 1, \quad 
L = \sup_M |u| + 1, \quad 
A = 2L(C_0 + 1),
\]
and $C_0$ is a uniform positive constant to be determined later. We have the following derivatives:
\begin{align}
(4K)^{-1} &\le \varphi' \le (2K)^{-1}, \quad \varphi'' = 2(\varphi')^2, \label{v1} \\[1mm]
\frac{1}{2} A L^{-1} &\le -\psi' \le A L^{-1}, \quad \psi'' = A^{-1} (\psi')^2.\label{v2}
\end{align}

Suppose $H$ attains its maximum at $(z_0, \xi_0)$. We choose holomorphic coordinates near $z_0$ such that
\[
\begin{aligned}
&g_{i\bar{j}}(z_0) = \delta_{ij}, \quad \partial_k g_{i\bar{j}}(z_0) = 0, \quad \text{and} \\
&w_{i\bar{j}}(z_0) = \lambda_i \delta_{ij} \quad \text{with} \quad \lambda_1 \ge \lambda_2 \ge \dots \ge \lambda_n,
\end{aligned}
\]
which forces $\xi_0 = \frac{\partial}{\partial z_1}$. Extend $\xi_0$ near $z_0$ by $\tilde{\xi}_0(z) = (g_{1\bar{1}})^{-\frac 12} \frac{\partial}{\partial z_1}$. Then the function
\[
Q(z) = H(z, \tilde{\xi}_0(z))
\]
attains a local maximum at $z_0$.

By the maximum principle, at $z_0$ we have
\begin{equation}\label{firstordercondition}
Q_i = \frac{w_{1\bar{1},i}}{w_{1\bar{1}}} + \varphi_i + \psi_i = 0,
\end{equation}
and
\begin{align}\label{LQ}
 		0\ge \mathcal{L}(Q)=\sum_{i}F^{i\bar i}Q_{i\bar i}
 		=&
 	\sum_{i}	\lambda_1^{-1}F^{i\bar i}w_{1\bar 1, i\bar i}-\lambda_1^{-2} \sum_{i}F^{i\bar i}{| w_{1\bar 1, i}|^2}\notag\\
 	&+	\sum_{i}\varphi''F^{i\bar i}\big|( |\nabla u|^2)_i\big|^2+\varphi'\sum_{i}F^{i\bar i} \big(|\nabla u|^2\big)_{i\bar i}\notag\\
 	&+\psi''\sum_{i}F^{i\bar i}|u_{i}|^2+\psi'\sum_{i}F^{i\bar i}u_{i\bar i}.
 	\end{align}

We will estimate each term in \eqref{LQ}. In what follows, we always assume
\begin{equation}\label{lambda1large}
\lambda_1 > 80 A \sqrt{A} K,
\end{equation}
otherwise the desired estimate holds trivially.

Since $d\chi = 0$, the commutation formula for covariant derivatives yields
\begin{align}
w_{1\bar{1}, i\bar{i}} = w_{i\bar{i}, 1\bar{1}} + (\lambda_1 - \lambda_i) R_{1\bar{1}i\bar{i}}.\label{covariantderivativeformulae}
\end{align}
By \eqref{differentiate2} and the fact that $\lambda_i + \lambda_j \ge 0$ for $i \neq j$ (Lemma \ref{yuanlemma}), we obtain
\begin{align}\label{Fiiwii11}
\sum_{i}F^{i\bar{i}} w_{i\bar{i}, 1\bar{1}} 
&= \sum_{i,j} \frac{\lambda_i + \lambda_j}{(1+\lambda_i^2)(1+\lambda_j^2)} |w_{i\bar{j},1}|^2 \notag \\
&\ge \sum_{i=2}^n \frac{\lambda_1 + \lambda_i}{(1+\lambda_1^2)(1+\lambda_i^2)} |w_{1\bar{1},i}|^2 
   + \sum_{i=1}^n \frac{2\lambda_i}{(1+\lambda_i^2)^2} |w_{i\bar{i},1}|^2,
\end{align}
where we have used $w_{i\bar{1},1} = w_{1\bar{1},i}$, which follows from $d\chi = 0$.

 	We now prove the following inequality
\begin{align}\label{3Fiiwii11}
    \sum_{i} \frac{\lambda_i}{(1+\lambda_i^2)^2} | w_{i\bar i, 1} |^2 \ge 0.
\end{align}
This is clearly true if $\lambda_n \ge 0$.  
Assume now that $\lambda_n < 0$. By \eqref{differentiate1}, we have
\[
\Bigl| \frac{w_{n\bar n, 1}}{1+\lambda_n^2} \Bigr| = \Bigl| \sum_{i=1}^{n-1} \frac{w_{i\bar i, 1}}{1+\lambda_i^2} \Bigr|.
\]
Using this relation and the fact that $\lambda_i > 0$ for $1 \le i \le n-1$, we obtain
\begin{align*}
\sum_{i=1}^n \frac{\lambda_i}{(1+\lambda_i^2)^2} | w_{i\bar i, 1} |^2
&= \sum_{i=1}^{n-1} \frac{\lambda_i |w_{i\bar i, 1}|^2}{(1+\lambda_i^2)^2}
   + \lambda_n \Bigl| \sum_{i=1}^{n-1} \frac{w_{i\bar i, 1}}{1+\lambda_i^2} \Bigr|^2 \\
&\ge \Bigl( \sum_{i=1}^{n-1} \lambda_i^{-1} \Bigr)^{-1} 
      \Bigl| \sum_{i=1}^{n-1} \frac{w_{i\bar i, 1}}{1+\lambda_i^2} \Bigr|^2
   + \lambda_n \Bigl| \sum_{i=1}^{n-1} \frac{w_{i\bar i, 1}}{1+\lambda_i^2} \Bigr|^2 \\
&= \frac{1 + \lambda_n \sum\limits_{i=1}^{n-1} \lambda_i^{-1}}
          {\sum\limits_{i=1}^{n-1} \lambda_i^{-1}}
      \Bigl| \sum_{i=1}^{n-1} \frac{w_{i\bar i, 1}}{1+\lambda_i^2} \Bigr|^2 \\
&\ge 0,
\end{align*}
where the last inequality follows from
\[
1 + \lambda_n \sum_{i=1}^{n-1} \lambda_i^{-1} 
= \lambda_n \, \frac{\sigma_{n-1}(\lambda)}{\sigma_n(\lambda)} \ge 0.
\]
 		  
 Combining \eqref{Fiiwii11} with \eqref{3Fiiwii11}, we obtain
\begin{align}
\sum_{i} F^{i\bar{i}} w_{i\bar{i}, 1\bar{1}} \ge 
\sum_{i=2}^n \frac{\lambda_1 + \lambda_i}{(1+\lambda_1^2)(1+\lambda_i^2)} | w_{1\bar{1}, i} |^2.
\end{align}
Consequently,
\begin{align}
&\lambda_1^{-1} \sum_{i}F^{i\bar{i}} w_{i\bar{i}, 1\bar{1}} 
- \lambda_1^{-2}\sum_{i} F^{i\bar{i}} | w_{1\bar{1},i} |^2 \notag \\
\ge& -\lambda_1^{-2} (1+\lambda_1^2)^{-1} | w_{1\bar{1}, 1} |^2  - \sum_{i=2}^n \frac{| w_{1\bar{1}, i} |^2}{\lambda_1^{2}(1+\lambda_1^2)(1+\lambda_i^2)} \notag \\
& + \sum_{i=2}^n \frac{\lambda_i}{\lambda_1(1+\lambda_1^2)(1+\lambda_i^2)} | w_{1\bar{1}, i} |^2 \notag \\
=& \mathcal{I} + \mathcal{II} + \mathcal{III}. \label{Fiiwii11w11i}
\end{align}

We first estimate $\mathcal{I}$ and $\mathcal{II}$ using \eqref{firstordercondition}:
\begin{align} \label{termI}
\mathcal{I} &\ge -2(\varphi')^2 F^{1\bar{1}} |(|\nabla u|^2)_1|^2 
               - 2(\psi')^2 F^{1\bar{1}} |u_{1}|^2 \notag \\
            &\ge -\varphi'' F^{1\bar{1}} |(|\nabla u|^2)_1|^2 
               - 2(\psi')^2 K F^{1\bar{1}},
\end{align}
and
\begin{align} \label{termII}
\mathcal{II} &= -(1+\lambda_1^2)^{-1} \sum_{i=2}^n \frac{F^{i\bar{i}} | w_{1\bar{1}, i} |^2}{\lambda_1^{2}} \notag \\
             &\ge -2(1+\lambda_1^2)^{-1} (\varphi')^2 
                    \sum_{i=2}^n F^{i\bar{i}} |(|\nabla u|^2)_i|^2 \notag \\
             &\quad - 2(1+\lambda_1^2)^{-1} (\psi')^2 
                    \sum_{i=2}^n F^{i\bar{i}} |u_{i}|^2 \notag \\
             &\ge -\frac{1}{12} \varphi'' \sum_{i=2}^n F^{i\bar{i}} |(|\nabla u|^2)_i|^2 
                - \frac{1}{2} \psi'' \sum_{i=2}^n F^{i\bar{i}} |u_{i}|^2,
\end{align}
where the last inequality follows from \eqref{lambda1large}.

Next we estimate the term $\mathcal{III}$. 
Since $\lambda_i>0$ for $1\le i\le n-1$ by Lemma \ref{yuanlemma}, we obtain
\[
\mathcal{III} \ge \frac{\lambda_n|w_{1\bar{1}, n}|^2}{\lambda_1(1+\lambda_1^2)(1+\lambda_n^2)}.
\]
Then, using \eqref{firstordercondition} and the Cauchy-Schwarz inequality, we have
\begin{align}\label{estimateIII1}
\mathcal{III} &\ge -\frac{4}{3}\frac{\lambda_1|\lambda_n|}{1+\lambda_1^2}(\varphi')^2F^{n\bar{n}}|(|\nabla u|^2)_n|^2 
                -\frac{4\lambda_1|\lambda_n|}{1+\lambda_1^2}(\psi')^2F^{n\bar{n}}|u_n|^2 \notag \\
              &\ge -\frac{2}{3}\varphi''F^{n\bar{n}}|(|\nabla u|^2)_n|^2 
                -\frac{4\lambda_1|\lambda_n|}{1+\lambda_1^2}(\psi')^2F^{n\bar{n}}|u_n|^2,
\end{align}
where we used \eqref{v1}.

Let $\tau = \frac{1}{8A}$. We proceed to estimate $\mathcal{III}$ in two cases.

\textbf{Case 1: $|\lambda_n| \le \tau \lambda_1$.} 
In this case, we obtain
\begin{align}\label{IIIcase1}
\mathcal{III} &\ge -\frac{2}{3}\varphi''F^{n\bar{n}}|(|\nabla u|^2)_n|^2 - 4\tau(\psi')^2F^{n\bar{n}}|u_n|^2 \notag \\
              &\ge -\frac{2}{3}\varphi''F^{n\bar{n}}|(|\nabla u|^2)_n|^2 - \frac{1}{2}\psi''F^{n\bar{n}}|u_n|^2,
\end{align}
where we used \eqref{v2}.

\textbf{Case 2: $|\lambda_n| \ge \tau \lambda_1$.} 
By \eqref{lambda1large}, we have 
$|u_{n\bar{n}}|^2 = (\lambda_n - \chi_{n\bar{n}} - t)^2 \ge \frac{\tau^2}{2}\lambda_1^2$,
and consequently
\[
4(\psi')^2\frac{\lambda_1|\lambda_n|}{1+\lambda_1^2}F^{n\bar{n}}|u_n|^2 
\le 4A F^{n\bar{n}} K 
\le \frac{1}{3} \varphi'F^{n\bar{n}} u_{n\bar{n}}^2.
\]
Inserting this inequality into \eqref{estimateIII1} yields
\begin{align}\label{IIIcase2}
\mathcal{III} \ge -\frac{2}{3}\varphi''F^{n\bar{n}}|(|\nabla u|^2)_n|^2 - \frac{1}{3}\varphi'F^{n\bar{n}} u_{n\bar{n}}^2.
\end{align}
	
 	Combining \eqref{IIIcase1} and \eqref{IIIcase2}, we obtain
\begin{align}\label{termIII}
    \mathcal{III} \ge 
    -\frac{2}{3}\varphi''F^{n\bar{n}}|(|\nabla u|^2)_n|^2 
    - \frac{1}{2}\psi''F^{n\bar{n}}|u_n|^2 
    - \frac{1}{3}\varphi'F^{n\bar{n}}u_{n\bar{n}}^2.
\end{align}

Inserting \eqref{termI}, \eqref{termII} and \eqref{termIII} into \eqref{Fiiwii11w11i}, we obtain
\begin{align}\label{Fiiwii11w11ifinal}
    &\lambda_1^{-1}\sum_{i} F^{i\bar{i}} w_{i\bar{i}, 1\bar{1}} 
      - \lambda_1^{-2} \sum_{i}F^{i\bar{i}} | w_{1\bar{1},i} |^2\notag \\
      &+ \frac{3}{4}\varphi'' \sum_{i}F^{i\bar{i}} |(|\nabla u|^2)_i|^2 
      + \psi'' \sum_{i}F^{i\bar{i}} |u_i|^2 
      + \frac{1}{3}\varphi' \sum_{i}F^{i\bar{i}} u_{i\bar{i}}^2 \notag \\
    \ge& \frac{1}{3}\varphi' F^{1\bar{1}} u_{1\bar{1}}^2 
      - 2K(\psi')^2 F^{1\bar{1}} > 0,
\end{align}
where we used $u_{1\bar{1}} \ge \frac{\lambda_1}{2}$, which follows from \eqref{lambda1large}.

Next we estimate $\sum\limits_{i}F^{i\bar{i}} (|\nabla u|^2)_{i\bar{i}}$. By direct computation, we have
\begin{align}\label{Fiinablausquareii}
    \sum_{i}F^{i\bar{i}}(|\nabla u|^2)_{i\bar{i}} = \sum_{i}F^{i\bar{i}}(u_{k\bar{i}}u_{\bar{k} i} + u_{ki}u_{\bar{k}\bar{i}}) 
    + \sum_{i,k}F^{i\bar{i}}(u_{k i\bar{i}}u_{\bar{k}} + u_{k}u_{\bar{k} i\bar{i}}).
\end{align}
Using the covariant derivative formulae and \eqref{differentiate1}, we obtain
\begin{align}
    &\sum_{i,k}F^{i\bar{i}}(u_{k i\bar{i}}u_{\bar{k}} + u_{k}u_{\bar{k} i\bar{i}}) \notag \\
    \ge &-\sum_{i,k}F^{i\bar{i}}(u_{i\bar{i} k}u_{\bar{k}} + u_{k}u_{i\bar{i}\bar{k}}) - |\text{Rm}||\nabla u|^2 \sum_{i} F^{i\bar{i}} \notag \\
    \ge &-2K(|\nabla \chi| + |\text{Rm}|) \sum_{i} F^{i\bar{i}}.
\end{align}
Inserting the above estimate into \eqref{Fiinablausquareii}, we get
\begin{align}\label{Fiinablausquareii2}
    \sum_{i}F^{i\bar{i}}(|\nabla u|^2)_{i\bar{i}} \ge \sum_{i}F^{i\bar{i}} u_{i\bar{i}}^2 
    - 2K(|\nabla \chi| + |\text{Rm}|) \sum_{i} F^{i\bar{i}}.
\end{align}

Inserting \eqref{covariantderivativeformulae}, \eqref{Fiiwii11w11ifinal} and \eqref{Fiinablausquareii2} into \eqref{LQ}, we obtain
\begin{align}\label{FQfinal}
    0 \ge \sum_{i} F^{i\bar{i}} Q_{i\bar{i}} 
    &\ge \psi' \sum_{i} F^{i\bar{i}} u_{i\bar{i}} 
      + \frac{1}{3}\varphi' \sum_{i}F^{i\bar{i}} u_{i\bar{i}}^2 
      - 3K \varphi'(|\nabla \chi| + |\text{Rm}|) \sum_{i} F^{i\bar{i}} \notag \\
    &\ge \psi' \sum_{i=1}^n F^{i\bar{i}} u_{i\bar{i}} 
      + \frac{1}{3}\varphi' \sum_{i}F^{i\bar{i}} u_{i\bar{i}}^2 
      - 2(|\nabla \chi| + |\text{Rm}|) \sum_{i} F^{i\bar{i}}.
\end{align}

By Lemma \ref{keylemma1} and noting that we have assumed $\underline{u}=0$, one of the following holds:
\begin{align*}
    &(i) \ -\sum_{i} F^{i\bar{i}} u_{i\bar{i}} \ge \delta_0 \sum_{i} F^{i\bar{i}}, \quad \text{or} \\
    &(ii) \ F^{1\bar{1}} \ge \delta_0 \sum_{i} F^{i\bar{i}}.
\end{align*}
We choose $C_0 = 6\delta_0^{-1} (\max |\nabla \chi| + |\text{Rm}|)$.

If ($i$) holds, then by \eqref{FQfinal} and $-\psi' \ge C_0 + 1$, we obtain a contradiction.
Thus ($ii$) holds, and we get $|\lambda_n| \ge \frac{\delta_0}{2} \lambda_1$. Consequently, we have
\begin{align*}
    0 &\ge \frac{\lambda_1^2}{4} F^{n\bar{n}} - 3K (|\nabla \chi| + |\text{Rm}|) \sum_{i} F^{i\bar{i}} \\
      &\ge \Bigl( \frac{\lambda_1^2}{4} - 3n K (|\nabla \chi| + |\text{Rm}|) \Bigr) F^{n\bar{n}}.
\end{align*}
Therefore, $\lambda_1 \le C K$, where $C$ is a uniform positive constant independent of $t$.
 \end{proof}

\section{Gradient estimates and the proof of the main theorems}
	
In this section, following the argument by Dinew and Ko{\l}odziej \cite{dk2017ajm} and Sz\'ekelyhidi \cite{szek2018},
we prove a Liouville theorem for the equation $\sigma_{n}+\sigma_{n-1}=0$. Then, based on the uniform complex Hessian estimates of the Hou-Ma-Wu type from Section 3, we get the gradient estimate  via a blow-up argument. Finally, we prove the main theorems.

We begin by defining viscosity solutions for the equation
\begin{align}
    \sigma_{n-1}(\sqrt{-1}\partial\bar\partial v) + \sigma_{n}(\sqrt{-1}\partial\bar\partial v) = 0. \label{vhq}
\end{align}
A $C^2$ function $v: \Omega \rightarrow \mathbb{R}$ is said to be \emph{$k$-admissible} in a domain $\Omega \subset \mathbb{C}^n$ if, for all $z \in \Omega$, the eigenvalue vector $\lambda(\sqrt{-1}\partial\bar\partial v(z))$ belongs to $\overline\Gamma_k$.

\begin{definition}
    $(i)$ A continuous function $v : \Omega \rightarrow \mathbb{R}$ is a viscosity subsolution of \eqref{vhq} if for every $C^2$ function $\varphi$ and every local minimum point $z_0 \in \Omega$ of $\varphi - v$, the following inequality holds:
    \begin{align*}
        \sigma_{n-1}(\sqrt{-1}\partial\bar{\partial} \varphi(z_0)) + \sigma_{n}(\sqrt{-1}\partial\bar{\partial} \varphi(z_0)) \ge 0.
    \end{align*}

    $(ii)$ A continuous function $v : \Omega \rightarrow \mathbb{R}$ is a viscosity supersolution of \eqref{vhq} if for every $(n-1)$-admissible function $\varphi$ and every local maximum point $z_0 \in \Omega$ of $\varphi - v$, the following inequality holds:
    \begin{align*}
        \sigma_{n-1}(\sqrt{-1}\partial\bar{\partial} \varphi(z_0)) + \sigma_{n}(\sqrt{-1}\partial\bar{\partial} \varphi(z_0)) \le 0.
    \end{align*}

    $(iii)$ A continuous function $v : \Omega \rightarrow \mathbb{R}$ is a viscosity solution of \eqref{vhq} if it is both a viscosity subsolution and a viscosity supersolution.
\end{definition}

If $v$ is a $C^2$ viscosity subsolution of \eqref{vhq}, then it satisfies $$\sigma_{n-1}(\sqrt{-1}\partial \bar \partial v) + \sigma_{n}(\sqrt{-1}\partial \bar \partial v) \ge 0$$ in the classical sense.

We now establish some basic properties.

\begin{lemma}\label{Lemma4.1}
    $(i)$ Let $\{v^i\}$ be a sequence of smooth viscosity subsolutions of \eqref{vhq} in $\Omega$ that converges locally uniformly to $v$. Then $v$ is also a viscosity subsolution of \eqref{vhq} in $\Omega$.

    $(ii)$ Let $v$ be a viscosity subsolution of \eqref{vhq} in $\mathbb{C}^n$. If $v$ is independent of $z_n$, then its restriction to $\mathbb{C}^{n-1}$ is a viscosity subsolution of $\sigma_{n-1}=0$.
\end{lemma}

 	\begin{proof}
 		For any $C^2$ function $\varphi$ and any local minimum point $z_0\in\Omega$ of $\varphi-v$,
 	we  need to show 
 	\begin{align}
 		\sigma_{n-1}(\varphi_{\alpha\bar \beta}(z_0))+\sigma_n(\varphi_{\alpha\bar \beta}(z_0))\ge 0.\label{vsubsolution}
 	\end{align}
By local uniform convergence, we can fix $r>0$ such that $v^i\to v$ uniformly on $\overline{B_r(z_0)}$. 
 Then for any $\epsilon\in (0,r)$, for all sufficiently large $i$, we have  $\sup_{B_{r}(z_0)}|v^i-v|<\frac{\epsilon^3}{3}$.
Then the function $\Phi = \varphi + \epsilon|z - z_0|^2 - v^i$ satisfies $\Phi|_{\partial B_{\epsilon}(z_0)} > \Phi(z_0)$, and thus $\Phi$ attains its minimum on $\overline{B_{\epsilon}(z_0)}$ at some point $z_{i,\epsilon} \in B_{\epsilon}(z_0)$.
 Since $v^i$ is a viscosity subsolution,
 	we have 
 	\begin{align*}
 		\sigma_{n-1}(\varphi_{\alpha\bar \beta}(z_{i,\epsilon})+\epsilon\delta_{\alpha\beta})+\sigma_n(\varphi_{\alpha\bar \beta}(z_{i,\epsilon})+\epsilon\delta_{\alpha\beta})\ge 0.
 	\end{align*}
Letting $\epsilon \to 0$ then implies \eqref{vsubsolution}. This proves $(i)$. 
 	
 	$(ii)$ follows directly from the definition of viscosity subsolution.
 	\end{proof}
 		
We will need the following comparison principle in the proof of the Liouville Theorem \ref{LiouTh}.
\begin{lemma}\label{Lemma4.2}
    Let $\Omega \subset \mathbb{C}^n$ be a bounded domain. Suppose $w$ is an $(n-1)$-admissible function and a viscosity subsolution of \eqref{vhq} in $\Omega$, and $v$ is a viscosity supersolution of \eqref{vhq} in $\Omega$. If $w \le v$ on $\partial\Omega$, then $w \le v$ in $\Omega$.
\end{lemma}

\begin{proof}
    Assume that there exists a point $x_0 \in \Omega$ such that $w(x_0) > v(x_0)$. Choose $\epsilon > 0$ sufficiently small so that
    \[
    \epsilon < \frac{w(x_0) - v(x_0)}{\max_{z \in \overline{\Omega}} |z|^2}.
    \]
    Since $w \le v$ on $\partial\Omega$, we deduce the following chain of inequalities:
    \begin{align*}
        \max_{z \in \overline{\Omega}} \,(w + \epsilon|z|^2 - v)
        &\ge w(x_0) - v(x_0) \\
        &> \epsilon \cdot \max_{z \in \partial\Omega} |z|^2 \\
        &\ge \max_{z \in \partial\Omega} \,(w + \epsilon|z|^2 - v).
    \end{align*}
    This implies that the function $w + \epsilon|z|^2 - v$ attains its maximum at an interior point $z_0 \in \Omega$.

    Now, since $w$ is $(n-1)$-admissible, the function $w + \epsilon|z|^2$ is also $(n-1)$-admissible. As $v$ is a viscosity  supersolution,  
   by the definition of  the supersolution, we have
    \begin{align}
        \sigma_{n-1}((w_{\alpha\bar{\beta}}(z_0)) + \epsilon I_n) +  \sigma_n((w_{\alpha\bar{\beta}}(z_0)) + \epsilon I_n) \le 0. \label{eq:contradiction-start}
    \end{align}
    However, since $\sigma_k(w_{\alpha\bar{\beta}}(z_0)) \ge 0$ for $1 \le k \le n-1$, expanding the left-hand side of \eqref{eq:contradiction-start} gives:
    \begin{align*}
        \sigma_{n-1}(w_{\alpha\bar{\beta}}(z_0)) +  \sigma_n(w_{\alpha\bar{\beta}}(z_0)) \le -\epsilon^{n} < 0. 
    \end{align*}
    This contradicts the assumption that $w$ is a viscosity subsolution of \eqref{vhq}, thereby completing the proof.
\end{proof}

We recall the definition of $k$-subharmonic functions following Blocki \cite{blocki05}. Let $\mathcal{E} = \sum_{\alpha} \sqrt{-1} dz_{\alpha} \wedge d\bar{z}_{\alpha}$ denote the standard K\"ahler metric on $\mathbb{C}^n$. 
\begin{definition}
For $1 \le k \le n$, a continuous function $v: \Omega \rightarrow \mathbb{R}$ is called \emph{$k$-subharmonic} if $v$ is subharmonic and
\[
\sqrt{-1}\partial \bar{\partial} v \wedge \gamma_1 \wedge \cdots \wedge \gamma_{k-1} \wedge \mathcal{E}^{n-k} \ge 0 \quad \text{in } \Omega
\]
for all real $(1,1)$-forms $\gamma_1, \dots, \gamma_{k-1}$ with $\lambda(\gamma_i) \in \overline{\Gamma}_k$ for $1 \le i \le k-1$.
\end{definition}

If $v$ is $C^2$ and $k$-subharmonic, then $v$ is $k$-admissible. More generally, by Lemma 3.7 in \cite{Lu2013JFA}, $v$ is a viscosity subsolution of $\sigma_k = 0$ if and only if $v$ is $k$-subharmonic.

The following lemma, which generalizes the G\r{a}rding inequality, is particularly useful for $k$-subharmonic functions (see \cite[page 1444, (3.4)]{blocki05}).

\begin{lemma}\label{garding}
Let $v, v_1, \dots, v_{k-1}$ be continuous $k$-subharmonic functions. Then
\[
(\sqrt{-1}\partial\bar{\partial} v) \wedge (\sqrt{-1}\partial\bar{\partial} v_1) \wedge \cdots \wedge (\sqrt{-1}\partial\bar{\partial} v_{k-1}) \wedge \mathcal{E}^{n-k} \ge 0
\]
in the sense of currents. In particular,
\[
\sqrt{-1}\partial v \wedge \bar{\partial} v \wedge (\sqrt{-1}\partial\bar{\partial} v)^{k-1} \wedge \mathcal{E}^{n-k} \ge 0
\]
in the sense of currents.
\end{lemma}

We now define weak solutions in the sense of currents.

\begin{definition}
    Let $v$ be a continuous $(n-1)$-subharmonic function in $\mathbb{C}^n$ such that $v + |z_{n+1}|^2$ is $n$-subharmonic in $\mathbb{C}^{n+1}$. We say that $v$ is a \emph{weak solution} (respectively, \emph{weak subsolution}) of the equation $\sigma_n + \sigma_{n-1} = 0$ if it satisfies
    \[
    (\sqrt{-1}\partial\bar{\partial} v)^n + n(\sqrt{-1}\partial\bar{\partial} v)^{n-1} \wedge \mathcal{E} = 0 \quad (\text{respectively, } \ge 0)
    \]
    in the sense of currents.
\end{definition}

\begin{remark}
     Note that the blow-up solution $u^{\infty}$ constructed below satisfies all the conditions in the above theorem.
    On the other hand, if $v$ is of class $C^2$, then $\sigma_n(\sqrt{-1}\partial\bar\partial v)+ \sigma_{n-1}(\sqrt{-1}\partial\bar\partial v) = 0$ (respectively, $\ge 0$) is equivalent to $    (\sqrt{-1}\partial\bar{\partial} v)^n + n(\sqrt{-1}\partial\bar{\partial} v)^{n-1} \wedge \mathcal{E} = 0$ (respectively, $\ge 0$).
\end{remark}

Let $\eta \in C^{\infty}(\mathbb{C}^n, \mathbb{R})$ be a mollifier, defined by
\begin{equation*}
    \eta(z) = 
    \begin{cases}
        C_n e^{\frac{1}{|z|^2 - 1}}, & |z| < 1, \\
        0, & |z| \geq 1,
    \end{cases}
\end{equation*}
where the constant $C_n$ is chosen so that $\int_{\mathbb{C}^n} \eta(z) \, \mathcal{E}^n(z) = 1$. Similarly, let $\eta_0$ be a mollifier on $\mathbb{C}$.
 
For a function $w: \mathbb{C}^n \rightarrow \mathbb{R}$, we define its mollification by
\begin{align}\label{mollification}
    [w]_{r}(z) = \int_{\mathbb{C}^n} w(z + ry) \eta(y) \, \mathcal{E}^n(y).
\end{align}
For the function $\tilde{w}(\tilde{z}) = w(z) + |z_{n+1}|^2$ defined on $\mathbb{C}^{n+1}$, where $\tilde{z} = (z, z_{n+1})$, we define 
\[
\langle\tilde{w}\rangle_{r}(\tilde{z}) = \frac{1}{n+1} \int_{\mathbb{C}^{n+1}} \tilde{w}(\tilde{z} + r\tilde{y}) \tilde{\eta}(\tilde{y}) \, \widetilde{\mathcal{E}}^{n+1}(\tilde{y}),
\]
where $\tilde{\eta}(\tilde{y}) = \eta_0(y_{n+1})\eta(y)$ with $\tilde{y} = (y, y_{n+1})$. A direct computation shows that
\begin{align*}
    \langle\tilde{w}\rangle_{r}(\tilde{z}) = [w]_{r}(z) + |z_{n+1}|^2 + r^2 M,
    \end{align*}
 where
    $M := \int_{\mathbb{C}} |y_{n+1}|^2 \eta_0(y_{n+1}) \, \sqrt{-1} dy_{n+1} \wedge d\bar{y}_{n+1}$.

The following lemma establishes key properties of weak solutions to the equation $\sigma_n + \sigma_{n-1} = 0$.

\begin{lemma}\label{keylemmavr}
    Let $v$ be a weak solution (in the sense of currents) to $\sigma_{n} + \sigma_{n-1} = 0$. Assume $\tilde v (\tilde z) = v(z) + |z_{n+1}|^2$ is $n$-subharmonic in $\mathbb{C}^{n+1}$.
    Then the following properties hold:
    \begin{enumerate}[label=(\roman*)]
        \item The function $\tilde{v}(\tilde{z}) $ is a weak solution to $\sigma_{n}=0$ and $\langle\tilde{v}\rangle_{r}$ is $n$-subharmonic in $\mathbb{C}^{n+1}$.
        \item The mollification $[v]_{r}$ is a smooth subsolution to the same equation.
        \item If $0\le v\le 1$,  $\frac{1}{2}v^2$ is a subsolution to the same equation.
        \item {If $w$ is a subsolution to the same equation, then $\frac{1}{2}(v + w)$ is also a subsolution.}
    \end{enumerate}
\end{lemma}

 \begin{proof}
We first note that $[v]_{r}$ is $(n-1)$-subharmonic since $v$ is $(n-1)$-subharmonic. To prove $(i)$, we will show that $\tilde{v}$ satisfies the equation: $\sigma_{n}=0$ on $\mathbb{C}^{n+1}$ in the current sense.

Let $e_{n+1} = \sqrt{-1}dy_{n+1} \wedge d\bar{y}_{n+1}$ and let $\widetilde{\mathcal{E}} = e_{n+1}+ \mathcal{E}$ be the standard K\"ahler metric on $\mathbb{C}^{n+1}$. For any $1 \le k \le n$, we have:
\begin{align*}
   & (\sqrt{-1}\partial\bar{\partial} \tilde{v})^{k} \wedge \widetilde{\mathcal{E}}^{n+1-k}\\
    =& C^k_{n+1} (\sqrt{-1}\partial\bar{\partial} v + e_{n+1})^{k} \wedge (\mathcal{E} + e_{n+1})^{n+1-k} \\
    =& \left((\sqrt{-1}\partial\bar{\partial} v)^{k} + k(\sqrt{-1}\partial\bar{\partial} v)^{k-1} \wedge e_{n+1}\right) \wedge \left(\mathcal{E}^{n-k+1} + (n+1-k)\mathcal{E}^{n-k} \wedge e_{n+1}\right) \\
    =& \left((n-k+1)(\sqrt{-1}\partial\bar{\partial} v)^{k} \wedge \mathcal{E}^{n-k} + k(\sqrt{-1}\partial\bar{\partial} v)^{k-1} \wedge \mathcal{E}^{n-k+1}\right) \wedge e_{n+1}.
\end{align*}
In particular, $(\sqrt{-1}\partial\bar\partial\tilde v)^n\wedge\mathcal{\tilde E }=0$ since $v$ solves
\[
(\sqrt{-1}\partial\bar{\partial} v)^{n} + n(\sqrt{-1}\partial\bar{\partial} v)^{n-1} \wedge \mathcal{E} = 0.
\]
Since $\tilde v$ is $n$-subharmonic, $\langle\tilde{v}\rangle_{r}$ is also $n$-subharmonic in $\mathbb{C}^{n+1}$.
 Note that $\langle\tilde{v}\rangle_{r} = [v]_{r}(z) + |z_{n+1}|^2 + r^2 M$; it follows that $[v]_{r}$ is a subsolution to the equation $\sigma_n + \sigma_{n-1} = 0$. This establishes $(ii)$.

To prove $(iv)$, let $W = \frac{v + w}{2}$ and $\widetilde{W}(\tilde{z}) = W(z) + |z_{n+1}|^2$. By $(i)$,  $\widetilde{W} = \frac{\tilde{v} + \tilde{w}}{2}$ is $n$-subharmonic in $\mathbb{C}^{n+1}$ and thus $W$ is a subsolution.

We now prove $(iii)$ using $(i)$. Define $\tilde{F}(\tilde{z}) = \frac{1}{2}v(z)^2 + |z_{n+1}|^2$. Then for any $1 \le k \le n$, we have:
\begin{align*}
&(\sqrt{-1}\partial\bar{\partial} \tilde{F})^{k} \wedge \widetilde{\mathcal{E}}^{n+1-k} \\
=& \left(v\sqrt{-1}\partial\bar{\partial} \tilde{v} + (1 - v)e_{n+1} + \sqrt{-1}\partial v \wedge \bar{\partial} v\right)^k \wedge \widetilde{\mathcal{E}}^{n+1-k} \\
=& v^k (\sqrt{-1}\partial\bar{\partial} \tilde{v})^k \wedge \widetilde{\mathcal{E}}^{n+1-k} \\
&+ k v^{k-1} (\sqrt{-1}\partial\bar{\partial} \tilde{v})^{k-1} \wedge \left((1 - v)e_{n+1} + \sqrt{-1}\partial v \wedge \bar{\partial} v\right) \wedge \widetilde{\mathcal{E}}^{n+1-k} \\
&+ k(k-1)(1 - v) v^{k-2} (\sqrt{-1}\partial\bar{\partial} \tilde{v})^{k-2} \wedge e_{n+1} \wedge (\sqrt{-1}\partial v \wedge \bar{\partial} v) \wedge \widetilde{\mathcal{E}}^{n+1-k}.
\end{align*}
Since $\tilde{v}$ is $n$-subharmonic in $\mathbb{C}^{n+1}$ and $0 \le v \le 1$, the above  is nonnegative by Lemma \ref{garding}.
\end{proof}

We immediately conclude that a weak solution (subsolution) in the sense of currents is also a viscosity solution (subsolution).

\begin{lemma}
    Let $v$ be a weak subsolution in the sense of currents to the equation $\sigma_{n} + \sigma_{n-1} = 0$. Then $v$ is also a viscosity subsolution.
\end{lemma}

\begin{proof}
   Since $\tilde{v} = v + |z_{n+1}|^2$ is $n$-subharmonic in $\mathbb{C}^{n+1}$,
it follows from Lemma 3.7 \cite{Lu2013JFA} that $\tilde{v}$ is a viscosity subsolution to $\sigma_n = 0$ in $\mathbb{C}^{n+1}$.
    
    Now, let $\varphi$ be any $C^2$ function and let $z_0$ be a local minimum point of $\varphi - v$. Consider the function $\tilde{\varphi} = \varphi + |z_{n+1}|^2$. Then $(z_0, 0)$ is a local minimum point of $\tilde{\varphi} - \tilde{v}$. Since $\tilde{v}$ is a viscosity subsolution to $\sigma_n = 0$, we have
    \[
    \sigma_n(\sqrt{-1}\partial\bar{\partial} \tilde{\varphi})(z_0, 0) \ge 0.
    \]
    Observing that
    \[
    \sigma_n(\sqrt{-1}\partial\bar{\partial} \tilde{\varphi})(z_0, 0) = \sigma_n(\sqrt{-1}\partial\bar{\partial} \varphi)(z_0) + \sigma_{n-1}(\sqrt{-1}\partial\bar{\partial} \varphi)(z_0),
    \]
    we conclude that $v$ is a viscosity subsolution.
\end{proof}

 Based on the above lemmas and following arguments similar to those in 
 \cite{dk2017ajm, szek2018}
 , we prove the following Liouville-type theorem. 
  Since in this case the equation is not homogeneous, we also need to use the complex Hessian bound of $[v]_{r}$.
\begin{theorem}\label{LiouTh}
    Let $v: \mathbb{C}^n \rightarrow \mathbb{R}$ be a $C^{0,1}$ function that is $(n-1)$-subharmonic and $\tilde v=v+|z_{n+1}|^2$ is $n$-subharmonic. Assume there exists a positive constant $C$ with $\|v\|_{C^1(\mathbb{C}^n)} \le C$ and $\Delta v \le C$ in the weak sense (see Remark~\ref{vlaplacebound} for details). If $v$ is a weak solution in the sense of currents and a viscosity supersolution to the equation
    \begin{align}\label{nn-1equation}
        \sigma_{n-1}(\sqrt{-1}\partial\bar{\partial} v) + \sigma_{n}(\sqrt{-1}\partial\bar{\partial} v) = 0,
    \end{align}
    then $v$ must be a constant.
\end{theorem}

\begin{remark}\label{vlaplacebound}
    The condition $\Delta v \le C$ in the weak sense means that there exists a constant $C > 0$ such that for every nonnegative test function $f \in C_c^{\infty}(\mathbb{C}^n)$ with $\int_{\mathbb{C}^n} f(z) \, \mathcal{E}^n(z) = 1$, we have
    \[
    \int_{\mathbb{C}^n} v \, \sqrt{-1}\partial\bar{\partial} f \wedge \mathcal{E}^{n-1} \le C.
    \]
    From this condition, it follows immediately that the regularization $[v]_r$ has bounded Laplacian and hence bounded complex Hessian.
\end{remark}
 	
\begin{proof}[{Proof of Theorem \ref{LiouTh}}]
    Assume that $v$ is not constant. Then $\inf_{\mathbb{C}^n} v < \sup_{\mathbb{C}^n} v$.
    By scaling, we may assume that $v$ satisfies the following normalized conditions:
    \begin{equation*}
        \begin{cases}
            \sup_{\mathbb{C}^n} v = 1, \quad \inf_{\mathbb{C}^n} v = 0, \\
            |\nabla v| \le C_1 := (\sup_{\mathbb{C}^n} v - \inf_{\mathbb{C}^n} v)^{\frac{1}{2}}, \quad \Delta v \le C, \\
            \sigma_{n-1}(\sqrt{-1}\partial\bar{\partial} v) + \sigma_{n}(\sqrt{-1}\partial\bar{\partial} v) = 0.
        \end{cases}
    \end{equation*}
    Indeed, the function $\tilde{v}(z) = C_{1}^{-2}({v(C_1z) - \inf_{\mathbb{C}^n} v})$ satisfies these conditions. 
   
    For sufficiently small $\epsilon > 0$, let $v^{\epsilon} = [v]_{\epsilon}$. 

 \textbf{Case 1}: There exist a positive constant $\rho$, sequences $\epsilon_k \to 0$, $x_k \in \mathbb{C}^n$, $r_k \to \infty$, and a sequence of unit $(1,0)$-vectors $\xi_k$ such that
    \begin{align}
        \frac{1}{2}\Bigl(\Bigl[\frac{(v^{\epsilon})^2}{3}\Bigr]_{r}(x_k) + [v]_{\rho}(x_k) - \frac{7}{6}\Bigr) &\ge v(x_k) \label{w-v} 
\end{align}
\begin{align}
        \lim_{k \to \infty} \int_{B_{r_k}(x_k)} |v_{\xi_k}^{\epsilon}|^2(z) \, \mathcal{E}^n(z) &= 0. \label{integral-zero}
    \end{align}
    Following the arguments in \cite{szek2018} and \cite{dk2017ajm}, we obtain an $(n-1)$-subharmonic function $v^{\infty}$ that is also a solution to the same equation and is independent of $z_n$. This implies that $v^{\infty}$ is plurisubharmonic in $\mathbb{C}^{n-1}$. Since $v^{\infty}$ is bounded, it must be constant, contradicting \eqref{w-v}.

 \textbf{Case 2}: Suppose Case 1 does not hold. Then for every $\rho > 0$, there exists a constant $C_{\rho} > 0$ with the following property: for all $\epsilon < C_{\rho}^{-1}$, $r > C_{\rho}$, $x \in \mathbb{C}^{n}$, and any unit $(1,0)$-vector $\xi$, if $\frac{1}{2}\bigl(\bigl[\frac{(v^{\epsilon})^2}{3}\bigr]_{r}(x) + [v]_{\rho}(x) -\frac{7}{6}\bigr) - v(x) \ge 0$, then
    \begin{align}\label{secondcase}
        \int_{B_{r}(x)} |v^{\epsilon}_{\xi}|^2(z) \, \mathcal{E}^n(z) \ge c_0 := C_{\rho}^{-1}.
    \end{align}

   Define the open set
    \begin{equation*}
        \Omega := \Bigl\{z \in \mathbb{C}^n : v(z) <w:=\frac{1}{2}\Bigl(\Bigl[\frac{(v^{\epsilon})^2}{3}\Bigr]_{r}(z) -{\tau_0}|z|^2+ [v]_{\rho}(z) - \frac{7}{6}\Bigr)  \Bigr\},
    \end{equation*}
    where $\tau_0>0$ is a constant to be determined later.
    
    Since $0 \le v \le 1$, the set $\Omega$ is bounded. As $\inf_{\mathbb{C}^n} v = 0$, we may  translate so that $v(0) < \frac{1}{60}$. By Cartan's lemma (see \cite{dk2017ajm}), we choose $\rho$ large enough so that $[v]_{\rho}(0) > \frac{9}{10}$, and choose $r$ large enough so that $\bigl[\frac{(v^\epsilon)^2}{3}\bigr]_{r}(0) > \frac{3}{10}$. Then $0 \in \Omega$.

    We now show that $\frac{1}{2}\bigl(\bigl[\frac{(v^{\epsilon})^2}{3}\bigr]_{r}(z) -{\tau_0}|z|^2+ [v]_{\rho}(z) \bigr)$
     is a subsolution to $\sigma_{n-1} + \sigma_{n} = 0$ in $\Omega$. Assuming this, the comparison principle (Lemma \ref{Lemma4.2}) implies that $v \ge w$ in $\Omega$, contradicting the definition of $\Omega$.

    By Lemma \ref{keylemmavr} $(ii)$, we know that $[v]_{r}$, $[v]_{\rho}$, $\frac{(v^{\epsilon})^2}{2}$, and $\bigl[\frac{(v^{\epsilon})^2}{2}\bigr]_{r}$ are all subsolutions to \eqref{nn-1equation}.

    From \eqref{secondcase} and following \cite{dk2017ajm, szek2018}, there exists a positive constant $b_0$, depending only on $c_0$ and $C_n$, such that $\bigl[\frac{(v^{\epsilon})^2}{2}\bigr]_{r} - b_0|z|^2$ is $(n-1)$-subharmonic in $\Omega$. 
Then by the concavity of $\sigma_{n-1}^{1/(n-1)}$ in $\Gamma_{n-1}$, for any $z\in \Omega$, we have 
    \[
    \sigma_{n-1}^{1/(n-1)}\Bigl(\sqrt{-1}\partial\bar{\partial}\Bigl[\frac{(v^{\epsilon})^2}{2}\Bigr]_{r}\Bigr) \ge (n-1)^{1/(n-1)} b_0,
    \]
    and thus
    \begin{align}\label{b0bound}
        \sigma_{n-1}\Bigl(\sqrt{-1}\partial\bar{\partial}\Bigl[\frac{(v^{\epsilon})^2}{2}\Bigr]_{r}\Bigr) \ge (n-1) b_0^{n-1}.
    \end{align}

    We now show that $F := \bigl[\frac{(v^{\epsilon})^2}{3}\bigr]_{r} - \tau_0|z|^2$ is a subsolution to $\sigma_n + \sigma_{n-1} = 0$ in $\Omega$ for sufficiently small $\tau_0$. Indeed, by \eqref{b0bound}, we have:
    \begin{align*}
        &\sigma_n(\sqrt{-1}\partial\bar{\partial} F) + \sigma_{n-1}(\sqrt{-1}\partial\bar{\partial} F) \\
        &= \sigma_n\Bigl(\sqrt{-1}\partial\bar{\partial} \Bigl[\frac{(v^{\epsilon})^2}{3}\Bigr]_{r}\Bigr) + \sigma_{n-1}\Bigl(\sqrt{-1}\partial\bar{\partial} \Bigl[\frac{(v^{\epsilon})^2}{3}\Bigr]_{r}\Bigr) + O(\tau_0) \\
        &= \Bigl(\frac{2}{3}\Bigr)^n \Bigl(\sigma_n\Bigl(\sqrt{-1}\partial\bar{\partial} \Bigl[\frac{(v^{\epsilon})^2}{2}\Bigr]_{r}\Bigr) + \frac{3}{2} \sigma_{n-1}\Bigl(\sqrt{-1}\partial\bar{\partial} \Bigl[\frac{(v^{\epsilon})^2}{2}\Bigr]_{r}\Bigr)\Bigr)+ O(\tau_0) \\
        &\ge \frac{1}{2}\Bigl(\frac{2}{3}\Bigr)^{n} \sigma_{n-1}\Bigl(\sqrt{-1}\partial\bar{\partial} \Bigl[\frac{(v^{\epsilon})^2}{2}\Bigr]_{r}\Bigr) + O(\tau_0) \\
        &\ge \frac{1}{2}\Bigl(\frac{2}{3}\Bigr)^{n} (n-1) b_0^{n-1} - C\tau_0 > 0,
    \end{align*}
    provided $\tau_0 <\frac{b_0^{n-1}}{2C} \left(\frac{2}{3}\right)^{n}$. Here we have used the fact that $[(v^{\epsilon})^2]_{r}$ has uniform $C^{0,1}$-bounds and complex Hessian bounds.
    
    By Lemma \ref{keylemmavr} $(iv)$, we get 
    $\frac{1}{2}\bigl(\bigl[\frac{(v^{\epsilon})^2}{3}\bigr]_{r}(z) -{\tau_0}|z|^2+ [v]_{\rho}(z) \bigr)$ is a subsolution to the equation $\sigma_{n}+\sigma_{n-1}=0$ in $\Omega$.
\end{proof}

 	\begin{proof}[Proof of the gradient estimate]
 	We prove the gradient estimate by contradiction, following the arguments in \cite{dk2017ajm, szek2018}. Assume that there exist a sequence $t_i \to 0$ and corresponding solutions $u^{t_i}$ to the LYZ equation \eqref{LYZt} such that for points $p_i \in M$, we have
 	
\begin{align}
    \sup_M |\nabla u^{t_i}|_{\omega} = |\nabla u^{t_i}|_{\omega}(p_i) = M_i \to \infty \quad \text{as } i \to \infty.
\end{align}

Since $M$ is compact, by passing to a subsequence, we may assume that $\{p_i\}$ converges to a point $p_0 \in M$. We choose local coordinates $\{w_\alpha\}_{1 \le \alpha \le n}$ in a ball $B_r(p_0)$ such that $p_i \in B_{r/2}(p_0)$ and $\omega(p_0) = \sum_{\alpha} \sqrt{-1} dw_{\alpha} \wedge d\bar{w}_{\alpha}$.

For any $R > 0$, we can choose $i$ sufficiently large such that for all $z \in B_R(0)$, the point $p_i + \frac{z}{M_i}$ lies in $B_{2r/3}(p_0)$. We define the rescaled functions
\[
\hat{u}^{t_i}(z) = u^{t_i}\bigl(p_i + \frac{z}{M_i}\bigr).
\]
We adopt the following notation:
\begin{align*}
    \hat{u}^{t_i}_{\alpha}(z) = \frac{\partial \hat{u}^{t_i}}{\partial z_{\alpha}}(z), \quad
    u^{t_i}_{\alpha}(w) = \frac{\partial u^{t_i}}{\partial w_{\alpha}}(w), 
\end{align*}
\begin{align*}
    \hat{u}^{t_i}_{\alpha\bar{\beta}}(z) = \frac{\partial^2 \hat{u}^{t_i}}{\partial z_{\alpha} \partial \bar{z}_{\beta}}(z), \quad
    u^{t_i}_{\alpha\bar{\beta}}(w) = \frac{\partial^2 u^{t_i}}{\partial w_{\alpha} \partial \bar{w}_{\beta}}(w).
\end{align*}
Then we have the scaling relations:
\begin{align*}
    u^{t_i}_{\alpha}(w) = M_i \hat{u}^{t_i}_{\alpha}(z), \quad
    u^{t_i}_{\alpha\bar{\beta}}(w) = M_i^2 \hat{u}^{t_i}_{\alpha\bar{\beta}}(z).
\end{align*}

By the $C^0$-estimate and the complex Hessian estimate for $u^{t_i}$ (for large $i$), we obtain
\begin{align*}
    \sup_{B_R(0)} \left(|\hat{u}^{t_i}| + |\partial \hat{u}^{t_i}|_{\mathcal{E}} + |\Delta \hat{u}^{t_i}|_{\mathcal{E}}\right) \le C.
\end{align*}
By standard elliptic theory, we have the estimates
\begin{align*}
    |\hat{u}^{t_i}|_{C^{1, \frac{2}{3}}(B_{2R/3}(0))} \le C.
\end{align*}

Passing to a subsequence, we find that $\hat{u}^{t_i} \to u^{\infty}$ in $C^{1, \frac{1}{2}}(B_{R/2}(0))$. Using a diagonal subsequence argument, we conclude that $\hat{u}^{t_i} \to u^{\infty}$ uniformly on compact subsets of $\mathbb{C}^n$, and $u^{\infty}$ is a bounded $C^{1,\frac{1}{2}}$ function defined on $\mathbb{C}^n$.

Since $u^{t_i}$ solves \eqref{LYZt}, by Lemma \ref{yuanlemma}, we have $\sigma_k(\chi + t_i\omega + \sqrt{-1}\partial\bar{\partial} u^{t_i}) \ge 0$ for $1 \le k \le n-1$. It follows that $u^{\infty}$ is $(n-1)$-subharmonic.
By construction, we have $|\partial u^{\infty}(0)| = 1$, so $u^{\infty}$ is nonconstant.

Now, let
\[
a_0 = -\liminf_{i \to \infty} \tan\hat{\theta}(t_i) M_i^2.
\]
Recall that $\hat{\theta}(t) \in (\pi - 2Ct, \pi-Ct)$, so $a_0 \in [0, \infty]$. By passing to a further subsequence, we may assume
\begin{align}\label{a0}
    a_0 = -\lim_{i \to \infty} \tan\hat{\theta}(t_i) M_i^2.
\end{align}

 \textbf{Case 1}: $a_0 = +\infty$.
In this case, we show that $u^{\infty}$ is plurisubharmonic in $\mathbb{C}^{n}$. Since $u^{t_i}$ solves the LYZ equation \eqref{LYZt}, we have
\[
\chi + t_i\omega + \sqrt{-1}\partial\bar{\partial} u^{t_i} > \cot\hat{\theta}(t_i)\omega.
\]
Then in the coordinate ball $B_r(p_0)$,
\begin{align*}
    (u^{t_i}_{\alpha\bar{\beta}})_{n\times n} \ge (\cot\hat{\theta}(t_i) - C)g.
\end{align*}

Now fix $R > 0$ and $\epsilon > 0$. For sufficiently large $i$ and all $z \in B_R(0)$, we have
\begin{align*}
    (\hat{u}^{t_i}_{\alpha\bar{\beta}}(z))_{n\times n} \ge -M_i^{-2}(\cot\hat{\theta}(t_i) - C)g \ge -\epsilon I_n,
\end{align*}
where we have used the fact that $\lim_{i\to\infty} M_i^{-2}\cot\hat{\theta}(t_i) = 0$.
Therefore, $\hat{u}^{t_i} + \epsilon|z|^2$ is plurisubharmonic in $B_R(0)$. Since $\hat{u}^{t_i}$ converges uniformly to $u^{\infty}$ on compact sets, it follows that $u^{\infty} + \epsilon|z|^2$ is plurisubharmonic in $B_R(0)$. 

As $R > 0$ and $\epsilon > 0$ are arbitrary, we conclude that $u^{\infty}$ is plurisubharmonic in $\mathbb{C}^n$.
However, since $u^{\infty}$ is bounded and plurisubharmonic on $\mathbb{C}^n$, it must be constant. This contradicts the fact that $|\partial u^{\infty}(0)| = 1$.			
 		
 	\textbf{Case 2}: $0 \le a_0 < \infty$. Define
\[
A^{i}(z) = \left(g^{-1}(p_i + M_i^{-1}z) - I_n\right)(\hat{u}^{t_i}_{\alpha\bar{\beta}}(z))_{n\times n} + M_{i}^{-2}g^{-1}\chi(p_i + M_i^{-1}z) + t_iM_{i}^{-2} I_n.
\]
Then $A^{i}$ converges uniformly to $0$ on any compact subset of $\mathbb{C}^n$.

Since $u^{t_i}$ solves the LYZ equation \eqref{LYZt}, we have
\begin{align}\label{LYZti}
    \sigma_{n-1}\bigl(\bigl(\hat{u}_{\alpha\bar{\beta}}^{t_i}\bigr)_{n\times n} + A^{i}\bigr) = M_i^{2}\tan\hat{\theta}(t_i)\sigma_n\bigl(\bigl(\hat{u}_{\alpha\bar{\beta}}^{t_i}\bigr)_{n\times n} + A^{i}\bigr) + O(M_i^{-2}),
\end{align}
where the matrix $A^{i}$ is uniformly bounded
and by Lemma \ref{yuanlemma}, we have
\begin{align}
    \sigma_{k}\bigl(\bigl(\hat{u}_{\alpha\bar{\beta}}^{t_i}\bigr)_{n\times n} + A^{i}\bigr) > 0 \quad \text{for any } 1 \le k \le n-1.\label{gamman-1}
\end{align}

\textit{Subcase 2.1}: $a_0 = 0$. In this subcase, by \eqref{LYZti}, \eqref{gamman-1}, and the fact that $\hat{u}^{t_i}$ converges locally uniformly to $u^{\infty}$, we conclude that $u^{\infty}$ is $(n-1)$-subharmonic and satisfies the following $(n-1)$-Hessian equation in the sense of currents $\sigma_{n-1}(\sqrt{-1}\partial\bar{\partial} u^{\infty}) = 0.$  Hence, applying the Liouville theorem of Dinew-Ko{\l}odziej \cite{dk2017ajm}, we find that $u^{\infty}$ is constant, which is a contradiction.

\textit{Subcase 2.2}: $0 < a_0 < \infty$. Since $\hat{u}^{t_i}$ converges uniformly to $u^{\infty}$ on compact subsets of $\mathbb{C}^n$, we have that $u^{\infty} + a_0|z|^2$ is plurisubharmonic, $(n-1)$-subharmonic, $u^{\infty}+|z_{n+1}|^2$ is $n$-subharmonic and $u^{\infty}$ satisfies the equation $\sigma_n + a_0\sigma_{n-1} = 0$ in the sense of currents.

Moreover, for any test function $f \in C_c^{\infty}(\mathbb{C}^n)$ with $f \ge 0$ and $\int_{\mathbb{C}^{n}} f = 1$, we have
\begin{align*}
    \int_{\mathbb{C}^{n}} u^{\infty} \sqrt{-1}\partial\bar{\partial} f \wedge \mathcal{E}^{n-1}
    &= \lim_{i \to \infty} \int_{\mathbb{C}^{n}} \hat{u}^{t_i} \sqrt{-1}\partial\bar{\partial} f \wedge \mathcal{E}^{n-1} \\
    &= \lim_{i \to \infty} \int_{\mathbb{C}^{n}} f \sqrt{-1}\partial\bar{\partial} \hat{u}^{t_i} \wedge \mathcal{E}^{n-1} \\
    &\le C.
\end{align*}
This proves that $\Delta u^{\infty} \le C$ in the weak sense.

Next, we prove $u^{\infty}$ is a supersolution to the equation $\sigma_{n-1}+a_0\sigma_{n}=0$ in the viscosity sense.
For any $C^2$ function $\varphi$ with 
$\lambda(\varphi_{\alpha\bar \beta})\in \overline\Gamma_{n-1}$ and for any local maximum point $z_0$  of  $\varphi-u^{\infty}$,  we need to show 
{\vspace*{-0.2cm}}
\begin{align}
	\sigma_{n-1}\big(\varphi_{\alpha\bar \beta}( z_0)\big)+a_0\sigma_{n}\big(\varphi_{\alpha\bar \beta}( z_0)\big)\le 0.\label{supersolution}
\end{align}

If $\sigma_{n-1}(\varphi_{\alpha\bar \beta}( z_0))=0$,  then $\sigma_{n}(\varphi_{\alpha\bar \beta}( z_0))\le 0$ and thus \eqref{supersolution} holds.

So we assume $c_0=\sigma_{n-1}(\varphi_{\alpha\bar \beta}( z_0))>0$. Then for any $\epsilon>0$ small enough, we have 
\vspace*{-0.2cm}
\begin{align*}
	\lambda(\varphi_{\alpha\bar \beta}-2\epsilon\delta_{\alpha\beta})\in \Gamma_{n-1}\ \ \text{and}\ \
	\sigma_{n-1}(\varphi_{\alpha\bar \beta}-2\epsilon\delta_{\alpha\beta})>\frac{c_0}{2} \ \text{in}\ \overline B_{\epsilon}(z_0)\label{sigman-1positive}.
\end{align*}
For sufficiently large $i$ such that $|\hat u^{t_{i}} - u^{\infty}| < \frac{\epsilon^{3}}{2}$, 
consider the function 
\vspace*{-0.25cm}
\[
\Phi = \varphi - \epsilon|z - z_0|^2 - \hat u^{t_i}.
\]
Since for any $z \in \partial B_{\epsilon}(z_0)$, 
\begin{align*}
	\Phi(z)=& \varphi(z) - \hat{u}^{t_{i}}(z) - \epsilon^3\\ <&\varphi(z) - u^{\infty}(z) - \frac{\epsilon^3}{2} 
	\le  \varphi(z_0) - u^{\infty}(z_0) - \frac{\epsilon^3}{2}\\
	<& \varphi(z_0) - \hat u^{t_i}(z_0) =\Phi(z_0).
\end{align*}
Thus, $\Phi$ attains its maximum on 
$\overline{B}_{\epsilon}(z_0)$ at an interior point $\tilde z_{0} \in B_{\epsilon}(z_0)$.
	Then $(\Phi_{\alpha\bar \beta}( \tilde z_{0 }))_{n\times n}\le 0$, namely $(\varphi_{\alpha\bar \beta}( \tilde z_{0})-\epsilon \delta_{\alpha\beta})_{n\times n}\le (\hat u^{t_i}_{\alpha\bar \beta}( \tilde z_{0}))_{n\times n}$. Thus
$(\hat u^{t_i}_{\alpha\bar \beta}( \tilde z_{0})+A^{i}_{\alpha\bar\beta}(\tilde z_{0}))_{n \times n}\ge(\varphi_{\alpha\bar \beta}(\tilde z_{0})-2\epsilon \delta_{\alpha\beta})_{n\times n}$ and we obtain
{\vspace*{-0.3cm}}
$$
\sigma_{n-1}(\hat u^{t_i}_{\alpha\bar \beta}(\tilde z_{0})+A_{\alpha\bar \beta}(\tilde z_{0}))\ge \sigma_{n-1}\big(\varphi_{\alpha\bar \beta}( \tilde z_{0})-2\epsilon \delta_{\alpha\beta}\big)\ge \frac{c_0}{2}.
$$
By the concavity of $\frac{\sigma_{n}}{\sigma_{n-1}}$ in $\Gamma_{n-1}$, we get
{\vspace*{-0.3cm}}
\begin{align}
	&1-\tan\hat \theta(t_{i})M^2_{i}\frac{\sigma_n\big(
		\varphi_{\alpha\bar \beta}(\tilde z_{0})-2\epsilon\delta_{\alpha\beta}
		\big)}{\sigma_{n-1}\big(\varphi_{\alpha\bar \beta}(\tilde z_{0})-2\epsilon\delta_{\alpha\beta}\big)}\notag\\\le
	&	1-\tan\hat \theta(t_{i})M^2_{i}\frac{\sigma_n(
		\hat u^{t_{i}}_{\alpha\bar \beta}(\tilde z_{0})
		+A^{i}_{\alpha\bar \beta}(\tilde z_{0}))}{\sigma_{n-1} (
		\hat u^{t_{i}}_{\alpha\bar \beta}(\tilde z_{0})
		+A^{i}_{\alpha\bar \beta}(\tilde z_{0}))}\notag\\
	\le& c_0^{-1}CM_{i}^{-2},\label{super2}
\end{align}
where in the last inequality we use  \eqref{LYZti}.
Then we get
\begin{align}
	{\sigma_{n-1}\big(\varphi_{\alpha\bar \beta}(\tilde z_{0})-2\epsilon\delta_{\alpha\beta}\big)}-\tan\hat \theta(t_{i})M^2_{i}{\sigma_n\big(
		\varphi_{\alpha\bar \beta}(\tilde z_{0})-2\epsilon\delta_{\alpha\beta}
		\big)}\le c_{0}^{-1}CM_{i}^{-2}\sup|D^2\varphi|.
	\end{align}
Let $i\rightarrow\infty$ and then $\epsilon\rightarrow 0$  in the above, we obtain
\eqref{supersolution}. Hence $u^{\infty}$ is a supersolution to the equation $\sigma_{n}+a_0\sigma_{n-1}=0$. 

We can then apply the Liouville Theorem \ref{LiouTh} to $a_0^{-1}u^{\infty}$ to conclude that $a_0^{-1}u^{\infty}$ is constant, which again leads to a contradiction.
\end{proof}

\begin{proof}[Proof of Theorems \ref{FYZ} and \ref{FYZfamily}]
Once the gradient estimates for $u^t$ are established, we obtain the uniform bound
\[
\|u^t\|_{C^1} + \max_M |\partial\bar{\partial} u^t| \le C,
\]
where $C$ is independent of $t$. Following the argument of Collins--Jacob--Yau \cite{cjy2020}, we derive uniform $C^{2,\alpha}$ estimates for $u^t$, also independent of $t$. For any $\beta \in (0,\alpha)$, there exists a subsequence $\{u^{t_i}\}$ converging in $C^{2,\beta}$ to a limit $u^0 \in C^{2,\alpha}$. Thus, $u^0$ solves the critical LYZ equation \eqref{LYZ1}. By standard elliptic regularity theory, $u^0$ is smooth.
\end{proof}

 	\section{Examples: the 3-dimensional case and 4-dimensional case}
 	In this section, as corollaries of our main theorem, we solve the 2-Hessian equation in dimension 3 and a Hessian quotient equation in dimension 4 under weaker conditions than those in the previous works. 
 	
 	In  dimension 3 , since
 	\begin{align*}
 		(\chi_u+\sqrt{-1}\omega)^3=&\chi_u^3-3\chi_u\wedge\omega^2+\sqrt{-1}(3\chi_u^2\wedge\omega-\omega^3),
 	\end{align*}
 	the critical LYZ equation is exactly the  2-Hessian equation
 	$$3\chi_u^2\wedge\omega=\omega^3.$$
 
 Sun \cite{sun2017cpam} solved the  above equation under  the following conditions:
 		\begin{align*}
 	 3\int_M\chi^2\wedge\omega=\int_M\omega^3,\quad \text{and} \quad
 		\chi\in \Gamma_2 (M),
 		\end{align*}
 	with the following continuity method
 	\begin{align*}
 		3\chi_u^2\wedge\omega=\Big(t+(1-t)\frac{3\chi^2\wedge\omega}{\omega^3}\Big)\omega^3.
 	\end{align*}
 	
 		We consider a family of LYZ  equations \eqref{LYZt} 
 	\begin{align*}
 		3(\chi+t\omega+\sqrt{-1}\partial\bar\partial u^t)^2\wedge\omega=&
 		\tan\hat\theta(t)(\chi+t\omega+\sqrt{-1}\partial\bar\partial u^t)^3\\
 		&-3\tan\hat\theta(t)(\chi+t\omega+\sqrt{-1}\partial\bar\partial u^t)\wedge\omega^2+\omega^3.
 	\end{align*}
 	As a corollary of our theorem, we solve the 2-Hessian equation  under weaker conditions instead of  the usual condition $\chi\in \Gamma_2(M)$.
 	\begin{corollary}
 		Let $(M,\omega)$ be a compact K\"ahler manifold with dimension $n=3$ and $\chi$ be a real $(1,1)$-form satisfying the following conditions:
 		\begin{enumerate}[(i)]
 			\item
 				$3\int_M\chi^2\wedge\omega=\int_M\omega^3$,
 				$\int_M\chi^3<3\int_M \chi\wedge\omega^2\label{re<0}$;
 			\item
 			the subsolution condition
 			\begin{align*}
 				\chi\wedge\omega>0\ \ \text{as a  $(2,2)$-form}.
 			\end{align*}
 		\end{enumerate}
 		Then there exists a unique smooth  solution $u$ with $\sup_M u=0$ solving
 		\begin{align*}
 			\sigma_2(\chi_u)=1,\quad \chi_u\in \Gamma_{2}(M).
 		\end{align*}
 	\end{corollary}

\begin{proof}
    Condition $(i)$ gives that the principal argument of $\int_M(\chi+\sqrt{-1}\omega)^3$ is $\pi$. Condition $(ii)$ means that $0$ is a subsolution of the critical LYZ equation in dimension $3$. According to Theorem \ref{FYZ}, there exists a smooth solution to the critical LYZ equation in dimension $3$, which is exactly the $2$-Hessian equation $\sigma_2(\chi_u)=1$.
\end{proof}

\begin{remark}
    If $\chi\in \Gamma_2(M)$, then it satisfies the subsolution condition $(ii)$. However, the subsolution condition is strictly weaker than $\chi\in \Gamma_2(M)$. 
    We will prove that the integral condition $(i)$ is necessary.
\end{remark}

Let $\lambda=(\lambda_1,\lambda_2,\lambda_3)$ be the eigenvalues of $\chi_u$ with respect to $\omega$ such that $\lambda_1\ge \lambda_2\ge \lambda_3$. Since $\sigma_2(\chi_u)=1$, we have
\begin{align*}
    \sigma_3(\chi_u)\le \frac{1}{9}\sigma_1(\chi_u)\sigma_2(\chi_u)=\frac{1}{9}\sigma_1(\chi_u).
\end{align*}
Integrating by parts, we obtain
\begin{align*}
    \int_M\chi^3 &= \int_M\chi_u^3 \le\frac{1}{3} \int_M{\chi_u\wedge\omega^2}= \frac{1}{3} \int_M{\chi\wedge\omega^2} \\
    &< 3 \int_M{\chi\wedge\omega^2},
\end{align*}
which proves the inequality in condition $(i)$.

 When $n=4$, the critical LYZ equation is the Hessian quotient equation. In this case, the subsolution condition yields $3\chi^2\wedge\omega-\omega^3>0$ as a $(3,3)$-form. Under the additional assumption that $\chi\in \Gamma_3(M)$, Sun \cite{sun2017cpam} and Sz\'{e}kelyhidi \cite{szek2018} proved the existence result using different continuity methods.
    
As a corollary of our result, we prove an existence result for the Hessian quotient equation $\sigma_3(\chi_u)=\sigma_1(\chi_u)$ under conditions weaker than those previously required.
\begin{corollary}\label{4dHQ1}
    Let $(M, \omega)$ be a compact K\"ahler manifold of dimension $n=4$.
    Let $\chi$ be a real $(1,1)$-form satisfying the following conditions:
    \begin{enumerate}[(i)]
        \item
            $ \int_M\chi^3\wedge\omega=\int_M\chi\wedge\omega^3$,
            $ \int_M\chi\wedge\omega^3>0$, and
            $ \mathrm{Re}\int_M (\chi+\sqrt{-1}\omega)^4<0$;
        \item
            $3\chi^2\wedge\omega-\omega^3>0$ as a $(3,3)$-form.
    \end{enumerate}
    Then there exists a unique smooth function $u$ with $\sup_{M} u=0$ solving 
    \begin{align}\label{4dHessianq}
        \sigma_3(\chi_u)=\sigma_1(\chi_u),\quad
        \chi_u\in \Gamma_3(M).
    \end{align}
\end{corollary}

\begin{proof}
    Condition $(i)$ implies that the principal argument of $\int_M(\chi+\sqrt{-1}\omega)^4$ is $\pi$. To apply Theorem \ref{FYZ}, we only need to show that $0$ is a subsolution of the critical LYZ equation.
    
    By condition $(ii)$, we have $\sigma_{2}(\chi)>2$, which implies $\sigma_1^2(\chi) \ge \frac{8}{3}\sigma_2(\chi) > \frac{16}{3}$. Moreover, since there exists a point $z_0 \in M$ such that $\sigma_1(\chi(z_0)) = 4\frac{\int_M\chi\wedge\omega^3}{\int_M \omega^4} > 0$, we conclude that $\sigma_1(\chi) > 0$ on $M$.
    
    Let $\mu=(\mu_1,\mu_2,\mu_3, \mu_4)$ be the eigenvalues of $\chi$ with respect to $\omega$, ordered such that $\mu_1 \ge \mu_2 \ge \mu_3 \ge \mu_4$. We need to prove
    \begin{align}\label{4dsub}
        \arctan\mu_2+\arctan\mu_3+\arctan\mu_4 > \frac{\pi}{2}.
    \end{align}
    Indeed, condition $(ii)$ gives $\sigma_2(\mu|1) = \mu_2\mu_3 + \mu_2\mu_4 + \mu_3\mu_4 > 1$. Since $\sigma_2(\mu|1) > 0$ and $\mu \in \Gamma_2$, it follows that $\mu_2 + \mu_3 > 0$, and consequently
    \[
    \mu_4 > \frac{1 - \mu_2\mu_3}{\mu_2 + \mu_3}.
    \]
    Therefore,
    \[
    \arctan \mu_4 > \arctan\Bigl( \frac{1 - \mu_2\mu_3}{\mu_2 + \mu_3} \Bigr) = \frac{\pi}{2} - \arctan\mu_2 - \arctan\mu_3,
    \]
    which proves \eqref{4dsub}. Hence, $0$ is a subsolution of the critical LYZ equation.
\end{proof}
We next show that condition $(i)$ is necessary.

\begin{lemma}
    Let $u$ be the unique smooth solution to the Hessian quotient equation \eqref{4dHessianq}. Then $\chi$ satisfies condition $(i)$.
\end{lemma}
	
\begin{proof}
    The equality in condition $(i)$ follows directly from integrating the equation.
    Since $\sigma_1(\chi_u)>0$, we have
    \begin{align*}
        \int_{M}\chi\wedge \omega^3 = \int_{M}\chi_u\wedge \omega^3 > 0.
    \end{align*}
    Let $\lambda=(\lambda_1,\lambda_2,\lambda_3,\lambda_4)$ be the eigenvalues of $\chi_u$ with respect to $\omega$, ordered such that $\lambda_1 \ge \lambda_2 \ge \lambda_3 \ge \lambda_4$.
    
    By the MacLaurin inequality, we obtain
    \begin{align*}
        \sigma_2 \ge 6\frac{\sigma_3}{\sigma_1} = 6. 
    \end{align*}
    The inequality $\lambda_4\sigma_1(\lambda|4) \le \sigma_2(\lambda|4)$ yields
    \begin{align*}
        \sigma_2 = \sigma_2(\lambda|4) + \lambda_4\sigma_1(\lambda|4) \le 2\sigma_2(\lambda|4),
    \end{align*}
    and consequently,
    \begin{align*}
        \sigma_2(\lambda|4) \ge 3. 
    \end{align*}
    Using the identity $\lambda_4 = \dfrac{\sigma_3(\lambda|4) - \sigma_1(\lambda|4)}{-\sigma_2(\lambda|4) + 1}$, we compute
    \begin{align*}
        \sigma_{4} - \sigma_2 + 1 &= \lambda_{4}\sigma_3(\lambda|4) - \lambda_4\sigma_1(\lambda|4) - \sigma_2(\lambda|4) + 1 \\
        &= \frac{\bigl(\sigma_3(\lambda|4) - \sigma_1(\lambda|4)\bigr)^2}{-\sigma_2(\lambda|4) + 1} - \sigma_2(\lambda|4) + 1 \\
        &\le -2.
    \end{align*}
    Therefore,
    \begin{align*}
        \mathrm{Re}\int_M (\chi+\sqrt{-1}\omega)^4 &= \mathrm{Re}\int_M (\chi_u+\sqrt{-1}\omega)^4 \\
        &= \int_M (\sigma_4 - \sigma_2 + 1)\,\omega^4 < 0,
    \end{align*}
    which establishes the second inequality in condition $(i)$.
\end{proof}
 	
 	\section{Further discussion on the Liouville-type theorem}
 	In this section, we generalize the Liouville-type theorem in the previous section and list two problems.

 		\begin{theorem}\label{LiouTh0}
 		Let $v: \mathbb{C}^n \rightarrow \mathbb{R}$ be a $C^{0,1}$ function with $n\ge 2$. Assume that for some $2\le k\le n$, $v$ is $(k-1)$-subharmonic and that  $\tilde v=v+|z_{n+1}|^2$ is $k$-subharmonic in $\mathbb{C}^{n+1}$.  Suppose there exists a positive constant $C$ such that $\|v\|_{C^{0,1}(\mathbb{C}^n)} \le C$ and $\Delta v\le C$ in the weak sense. If $v$ is a weak solution in the sense of currents and a viscosity supersolution to the equation
 		\begin{align}\label{kk-1Eq}
 			\sigma_{k-1}(\sqrt{-1}\partial\bar{\partial} v) + \sigma_{k}(\sqrt{-1}\partial\bar{\partial} v) = 0,
 		\end{align}
 		then $v$ must be a constant.
 	\end{theorem}

 		\begin{remark}\label{kvlaplacebound}
 		Since $\tilde v$ is $2$-subharmonic, the weak Laplacian bound implies that the regularization $[v]_r$ of $v$ has bounded complex Hessian.
 	\end{remark}

 	\begin{proof}
 		 The proof is based on Lemma \ref{LiouTh}.  
 		Assume that $v$ is not a constant function. Then $\inf_{\mathbb{C}^n} v < \sup_{\mathbb{C}^n} v$.
 		By scaling, we may assume that $v$ satisfies 
 		\begin{equation*}
 			\begin{cases}
 				\sup_{\mathbb{C}^n} v = 1, \quad \inf_{\mathbb{C}^n} v = 0, \\
 				|\nabla v| \le CC_1, \quad \Delta v \le C,
 			\end{cases}
\end{equation*}
 where $C_1$ denotes $(\sup_{\mathbb{C}^n} v - \inf_{\mathbb{C}^n} v)^{\frac{1}{2}}$.
 		We prove the theorem by induction on $n-k$. 
 		
 		 When $n-k=0$, this is Theorem \ref{LiouTh}.  		
 		Assume the theorem holds for $n-k\le m$. We need to prove it holds for  $n-k=m+1$. We still consider two cases. 
 		
 		\textbf{Case 1}: There exist a constant $\rho>0$ and  sequences $\epsilon_l \to 0$, $x_l \in \mathbb{C}^n$, $r_l \to \infty$, and  unit $(1,0)$-vectors $\xi_l$ such that
 		\begin{align}
 			\frac{1}{2}\Bigl(\Bigl[\frac{(v^{\epsilon})^2}{3}\Bigr]_{r}(x_l) + [v]_{\rho}(x_l) - \frac{7}{6}\Bigr) &\ge v(x_l), \label{kw-v} 
 		\end{align}
 		and
 		\begin{align}
 			\lim_{l \to \infty} \int_{B_{r_l}(x_l)} |v_{\xi_l}^{\epsilon}|^2(z) \, \mathcal{E}^n(z) &= 0. \label{kintegral-zero}
 		\end{align}
 		In this case, we  obtain a $(k-1)$-subharmonic function $v^{\infty}$ in $\mathbb{C}^n$ and $v^{\infty}$ solves $\sigma_{k-1}+\sigma_{k}=0$ in $\mathbb{C}^{n-1}$, has bounded Laplacian and $v^{\infty}+|z_{n+1}|^2$ is $k$-subharmonic in $\mathbb{C}^{n}$. 
 		Since $(n-1)-k=m$, by induction,  $v^{\infty}$ must be a constant.
 		
	\textbf{Case 2}: Suppose Case 1 does not hold. Then for every $\rho > 0$, there exists a constant $C_{\rho} > 0$ with the following property: for all $\epsilon < C_{\rho}^{-1}$, $r > C_{\rho}$, $x \in \mathbb{C}^{n}$, and any unit $(1,0)$-vector $\xi$,  when  $\frac{1}{2}\bigl(\bigl[\frac{(v^{\epsilon})^2}{3}\bigr]_{r}(z) -{\tau_0}|z|^2+ [v]_{\rho}(z) \bigr) - v(x) \ge 0$, we have
		\begin{align}\label{ksecondcase}
			\int_{B_{r}(x)} |v^{\epsilon}_{\xi}|^2(z) \, \mathcal{E}^n(z) \ge c_0 := C_{\rho}^{-1}.
		\end{align}
 		As before, the open set
 		\begin{equation*}
 			\Omega := \Bigl\{z \in \mathbb{C}^n : v(z) <w:=\frac{1}{2}\Bigl(\Bigl[\frac{(v^{\epsilon})^2}{3}\Bigr]_{r}(z) -{\tau_0}|z|^2+ [v]_{\rho}(z) - \frac{7}{6}\Bigr)  \Bigr\}.
 		\end{equation*}
 		 is  nonempty and bounded.
 	
 		We will show  $\frac{1}{2}\bigl(\bigl[\frac{(v^{\epsilon})^2}{3}\bigr]_{r}(z) -{\tau_0}|z|^2+ [v]_{\rho}(z) \bigr)$
 		is a subsolution to $\sigma_{k-1} + \sigma_{k} = 0$ in $\Omega$. Assuming this, the comparison principle (Lemma \ref{Lemma4.2}) implies that $v \ge w$ in $\Omega$, contradicting the definition of $\Omega$.
 		
 		We still have $[v]_{r}$, $[v]_{\rho}$, $\frac{(v^{\epsilon})^2}{2}$, and $\bigl[\frac{(v^{\epsilon})^2}{2}\bigr]_{r}$ are all subsolutions to \eqref{kk-1Eq}.
 		
 		From \eqref{ksecondcase}, there exists a positive constant $b_0$, depending only on $c_0$ and $C(n,k)$, such that $\bigl[\frac{(v^{\epsilon})^2}{2}\bigr]_{r} - b_0|z|^2$ is $(k-1)$-subharmonic. By the concavity of $\sigma_{k-1}^{1/(k-1)}$ in $\Gamma_{k-1}$, we have
 		\begin{align}\label{kb0bound}
 			\sigma_{k-1}\Bigl(\sqrt{-1}\partial\bar{\partial}\Bigl[\frac{(v^{\epsilon})^2}{2}\Bigr]_{r}\Bigr) \ge C_{n}^{k-1}b_0^{k-1}.
 		\end{align}
 		
 		We now show that $F := \bigl[\frac{(v^{\epsilon})^2}{3}\bigr]_{r} - \tau_0|z|^2$ is a subsolution to $\sigma_k + \sigma_{k-1} = 0$ in $\Omega$ for sufficiently small $\tau_0$. Indeed, by \eqref{kb0bound}, we have:
\[
 		\begin{aligned}
 			&\sigma_k(\sqrt{-1}\partial\bar{\partial} F) + \sigma_{k-1}(\sqrt{-1}\partial\bar{\partial} F) \\
 			=& \sigma_k\Bigl(\sqrt{-1}\partial\bar{\partial} \Bigl[\frac{(v^{\epsilon})^2}{3}\Bigr]_{r}\Bigr) + \sigma_{k-1}\Bigl(\sqrt{-1}\partial\bar{\partial} \Bigl[\frac{(v^{\epsilon})^2}{3}\Bigr]_{r}\Bigr) + O(\tau_0) \\
 			=& \Bigl(\frac{2}{3}\Bigr)^k \Bigl[\sigma_k\Bigl(\sqrt{-1}\partial\bar{\partial} \Bigl[\frac{(v^{\epsilon})^2}{2}\Bigr]_{r}\Bigr) + \frac{3}{2} \sigma_{k-1}\Bigl(\sqrt{-1}\partial\bar{\partial} \Bigl[\frac{(v^{\epsilon})^2}{2}\Bigr]_{r}\Bigr)\Bigr] + O(\tau_0) \\
 			\ge& \frac{1}{2}\Bigl(\frac{2}{3}\Bigr)^{k} \sigma_{k-1}\Bigl(\sqrt{-1}\partial\bar{\partial} \Bigl[\frac{(v^{\epsilon})^2}{2}\Bigr]_{r}\Bigr) + O(\tau_0) \\
 			\ge& \frac{1}{2}\Bigl(\frac{2}{3}\Bigr)^kC_{n}^{k-1}b_0^{k-1}-C\tau_0\\
 			>&0
 		\end{aligned}
\]
 		where $O(\tau_0)=\sum_{j=1}^{k-1}C_{j}\tau^{k-j}\sigma_{j}\bigl(\sqrt{-1}\partial\bar{\partial} \bigl[\frac{(v^{\epsilon})^2}{2}\bigr]_{r}\bigr)$ and since  $[(v^{\epsilon})^2]_{r}$ has uniform $C^{0,1}$-bounds and complex Hessian bounds, 
 		we have	$|O(\tau_0)|\le C\tau_0$. Then  $F$ is a subsolution in $\Omega$ if $\tau_0$ is sufficiently small.
 		
 		Similar to Lemma \ref{keylemmavr} $(iv)$, we get 
 		$\frac{1}{2}\bigl(\bigl[\frac{(v^{\epsilon})^2}{3}\bigr]_{r}(z) -{\tau_0}|z|^2+ [v]_{\rho}(z) \bigr)$ is a subsolution to the equation $\sigma_{k}+\sigma_{k-1}=0$ in $\Omega$.
 	\end{proof}
 	
We conclude this section with the following open problems. 
 	
 		\begin{question}
 			Can the conditions in Theorem \ref{LiouTh} and Theorem \ref{LiouTh0} be weakened?
 		\end{question}

 		\begin{question}
Let $k,l$ be fixed integers with $1\le l<k\le n$.
 		Are there suitable conditions such that all solutions of the following equation 
 			\begin{align*}
 				\sigma_{k}(\sqrt{-1}\partial\bar{\partial} u) + \sigma_{l}(\sqrt{-1}\partial\bar{\partial} u) = 0
 			\end{align*}
 			are constants? (The case $l=k-1$ is already treated in Theorem~\ref{LiouTh0}.)
 			
 		\end{question}
 		
\bibliographystyle{plain}
\bibliography{2026.06.10cLYZ}

\end{document}